\documentclass[twocolumn]{autart}    
\usepackage{cite}
\usepackage{amsmath,amssymb,amsfonts}
\let\theoremstyle\relax

\usepackage{amsthm}
\usepackage{algorithm}
\usepackage{algpseudocode}
\usepackage{graphicx}
\usepackage{subfig}
\usepackage{hyperref}

\usepackage{textcomp}
\usepackage{xcolor}

\newtheorem{theorem}{Theorem}[section]

\newtheorem{corollary}{Corollary}[section]

\newtheorem{definition}{Definition}[section]
\newtheorem{example}{Example}[section]
\newtheorem{exercise}{Exercise}[section]
\newtheorem{lemma}{Lemma}[section]

\newtheorem{problem}{Problem}[section]
\newtheorem{proposition}{Proposition}[section]
\newtheorem{remark}{Remark}[section]

\newcommand{\bthm}{\begin{theorem}}
	\newcommand{\ethm}{\end{theorem}}
\newcommand{\blem}{\begin{lemma}}
	\newcommand{\elem}{\end{lemma}}
\newcommand{\bex}{\begin{example}}
	\newcommand{\eex}{\end{example}}
\newcommand{\beg}{\begin{exercise}}
	\newcommand{\eeg}{\end{exercise}}
\newcommand{\bprop}{\begin{proposition}}
	\newcommand{\eprop}{\end{proposition}}
\newcommand{\bplm}{\begin{problem}}
	\newcommand{\eplm}{\end{problem}}
\newcommand{\bmrk}{\begin{remark}}
	\newcommand{\emrk}{\end{remark}}
\newcommand{\bdfn}{\begin{definition}}
	\newcommand{\edfn}{\end{definition}}
\newcommand{\bcor}{\begin{corollary}}
	\newcommand{\ecor}{\end{corollary}}

\newcommand{\beq}{\begin{equation}}
	\newcommand{\eeq}{\end{equation}}
\newcommand{\beqm}{\begin{equation*}}
	\newcommand{\eeqm}{\end{equation*}}
\newcommand{\beqn}{\begin{eqnarray}}
	\newcommand{\eeqn}{\end{eqnarray}}
\newcommand{\beqnm}{\begin{eqnarray*}}
	\newcommand{\eeqnm}{\end{eqnarray*}}
\newcommand{\bs}{\begin{subequations}}
	\newcommand{\es}{\end{subequations}}

\newcommand{\bei}{\begin{itemize}}
	\newcommand{\eei}{\end{itemize}}
\newcommand{\bed}{\begin{description}}
	\newcommand{\eed}{\end{description}}
\newcommand{\bee}{\begin{enumerate}}
	\newcommand{\eee}{\end{enumerate}}
\newcommand{\bey}{\begin{array}}
	\newcommand{\eey}{\end{array}}
\newcommand{\bec}{\begin{center}}
	\newcommand{\eec}{\end{center}}

\newcommand{\A}{\mathcal{A}}
\newcommand{\B}{\mathcal{B}}
\newcommand{\C}{\mathcal{C}}
\newcommand{\D}{\mathcal{D}}
\newcommand{\G}{\mathcal{G}}

\begin{document}

\begin{frontmatter}
\title{A tensor phase theory with applications in multilinear control\thanksref{footnoteinfo}} 

\thanks[footnoteinfo]{This paper was not presented at any IFAC 
meeting. Corresponding author: Guofeng Zhang. Tel. +852 2766-6936, Fax +852 2764-4382}

\author[Chengdong Liu]{Chengdong Liu}\ead{cdliu24@m.fudan.edu.cn},    
\author[Yimin Wei]{Yimin Wei}\ead{ymwei@fudan.edu.cn},               
\author[polyu,polyuc]{Guofeng Zhang}\ead{guofeng.zhang@polyu.edu.hk}  

\address[Chengdong Liu]{School of Mathematical Sciences, Fudan University, Shanghai, P. R. of China}  
\address[Yimin Wei]{School of Mathematical Sciences and  Key Laboratory of Mathematics for Nonlinear Sciences, Fudan University, Shanghai, P. R. of China}             

\address[polyu]{Department of Applied Mathematics, The Hong Kong Polytechnic University, Hung Hom, Kowloon, Hong Kong SAR, China}
\address[polyuc]{Hong Kong Polytechnic University Shenzhen Research Institute, Shenzhen 518057, China}


\begin{abstract}                          
The purpose of this paper is to initiate a phase  theory for tensors under the Einstein product, and explore its applications in multilinear control systems. Firstly, the sectorial tensor decomposition for sectorial tensors is derived, which allows us to define phases for sectorial tensors. A numerical procedure for computing phases of a sectorial tensor is also proposed. Secondly, the maximin and minimax expressions for tensor phases are given,  which are used to quantify how close the phases of a sectorial tensor are to those of its compressions. Thirdly, the compound spectrum,  compound numerical ranges and compound angular numerical ranges  of two sectorial tensors $\A,\B$ are defined and characterized in terms of the  compound numerical ranges and compound angular numerical ranges of the sectorial tensors $\A,\B$. Fourthly,  it is shown that the angles of eigenvalues of the product of two sectorial tensors are upper bounded by the sum of their individual phases. Finally, based on  the tensor phase theory developed above, a tensor version of  the small phase theorem  is  presented, which can be regarded as a natural generalization of the  matrix case, recently proposed in Ref. \cite{chen2024phase}. The results offer powerful new tools for the stability and robustness analysis of  multilinear feedback control systems.
\end{abstract}

\end{frontmatter}

\section{Introduction}

While vectors and matrices are the cornerstones of linear algebra for modeling linear relations, they are often insufficient for the complexities of modern science and engineering. Many contemporary data structures and system interactions exhibit high-dimensional, multi-way characteristics that are inherently multilinear. Tensors, also called hypermatrices, as the natural higher-order generalization of vectors and matrices, have emerged as the ideal mathematical framework for representing such large-scale, complex data and for modeling these multilinear interactions \cite{kolda2009tensor, acichocki2009nonnegativematrixandtensor,qi2017tensor,qi2018tensor,chen2024tensor}. Unlike matrices, which force the flattening of high-dimensional data and risk a loss of structural information, tensors natively preserve the structure of data and dynamical systems, thereby more effectively capturing the complex relationships within the data and among the variables of the dynamical  system \cite{hackbusch2012tensor, comon2014tensor,chen2024tensor,DPD+25}.

Leveraging this expressive power, tensor theory and methods have been widely applied across a diverse range of fields including social and biological network analysis \cite{dotson2022deciphering,DPD+25}, signal and image processing \cite{acichocki2009nonnegativematrixandtensor, comon2014tensor}, quantum information and computation \cite{orus2019tensor,HQZ16,QZN18,WGZL21}, quantum control \cite{Z14,Z17,zhang2022linear}, scientific computing \cite{lu2013multilinear,nie2023moment}, and systems and control theory \cite{ chen2021multilinear,chen2024tensor,wang2025algebraic}. A prominent example is the recent extension of the classical Lotka-Volterra model to capture high-order species interactions, where tensors directly govern the system's evolutionary dynamics and stability \cite{letten2019mechanistic, cui2023species,cui2024discrete,cui2025metzler}. This trend is particularly evident in control theory, underscored by the recent establishment of a tensor version of the celebrated small gain theorem \cite{wang2025algebraic}, which highlights the growing importance of tensor analysis in multilinear control systems \cite{chen2021multilinear,chen2024tensor,wang2025algebraic}.

In matrix theory, the numerical range (also called the field of values) and the numerical radius have been cornerstones of mathematical and engineering analysis \cite{r1,bonsall1973numerical, halmos2012hilbert}. The numerical range of a matrix $ A \in \mathbb{C}^{n \times n} $, defined as $ W(A) = \{ x^*Ax : x \in \mathbb{C}^n, \|x\|=1 \} $, provides critical geometric insight into its eigenvalue distribution, system stability, and operator behavior \cite{roger1994topics, axelsson1994numerical}. The numerical radius, $ w(A) = \sup\{ |\lambda| : \lambda \in W(A) \} $, quantifies the operator's maximum energy amplification. Notably, in iterative algorithms, the numerical radius is often a more reliable predictor of convergence speed than the spectral radius, as it captures the operator's transient behavior, not just its asymptotic properties \cite{axelsson1994numerical, eiermann1993fields}. A significant recent development is the emergence of a novel ``phase theory'' for matrices and linear control systems, built upon the concepts of the numerical range and sectorial matrices \cite{chen2019phase, MCQ22}. This theory has culminated in a ``small phase theorem'' that provides necessary and sufficient conditions for the stability of negative feedback systems, offering a  new perspective based on ``phase'' that complements the classical ``gain'' (norm) analysis \cite{SZQ23, chen2024phase}.

Inspired by the broad utility of the matrix numerical range, researchers have begun to generalize these concepts to the tensor setting. Pioneering work by Ke et al. \cite{ke2016numerical} introduced the notion of the tensor numerical range, demonstrating that this generalization preserves many essential properties of its matrix counterpart and opens new avenues in  multilinear algebra.  Numerical range and numerical radius for even-order square tensors  under the Einstein product are introduced  in \cite{r1},  where the author proved that the numerical range is a convex set.  Subsequent studies have explored the tensor numerical radius to understand the boundedness of multilinear operators \cite{MR3172255}. However, existing research has largely focused on generalizing ``gain'' aspects, such as numerical radii \cite{r1}. A systematic \textbf{tensor phase theory}---capable of characterizing the ``directional'' or ``phase'' properties of multilinear operators---remains a largely unexplored frontier. This gap limits our understanding of the intrinsic geometric properties of multilinear systems and hinders the application of powerful tools, like a small phase theorem, to multilinear control systems.

This paper aims to fill this gap by developing a phase theory for tensors, thereby providing  a new framework for analyzing multilinear control systems. Our work is built upon a solid mathematical foundation, including a deep understanding of the tensor numerical range \cite{ke2016numerical, braman2010third,r1,r7}. The main contributions of this paper are summarized as follows:

\begin{enumerate}
    \item \textbf{Definition of Tensor Phase.} By means of the notion of the numerical range, we define sectorial tensors in Definition \ref{def:sectorial}. Then we establish the \textbf{sectorial tensor decomposition theorem} for sectorial tensors  (Theorem \ref{th2_3}), a fundamental result that enables the formal definition of tensor phases; see Definition \ref{def:phases}. A procedure for computing phases of sectorial tensors is also given in Algorithm 1.

    \item \textbf{Characterization of Phase Properties.} We prove a minimax and maximin result for the phases of a sectorial tensor in Lemma \ref{th3_1}. Then we  define compressions of tensors in Definition \ref{def:compression} and study how close they are to the original sectorial tensor in terms of their phases; see 
    Theorems \ref{th3_2} and \ref{thm:aug1_temp}.  
Compound spectrum of even-order square tensors is defined in Definition \ref{def:compound spectrum}, and characterized by  numerical ranges and angular numerical ranges  of tensors in  Theorems \ref{th4_2} and  \ref{th4_3}. These properties are used to study phases of products and sums of sectorial tensors in Theorems \ref{th5_1} and \ref{th6_1}. In particular, in Theorem \ref{th5_1},  an upper bound of eigenvalues of a product of two sectorial tensors is given in terms of tensor phases. Rank robustness of sectorial tensors is investigated in Theorem \ref{th7_1}. These results mirror and extend classical results from matrix analysis to the multilinear case.

    \item \textbf{Establishment of a Small Phase Theorem for Tensors.} As a culminating application, we prove a \textbf{small phase theorem} for multilinear feedback systems; see Theorem \ref{thm:small_phase_theorem}. This theorem provides a new stability condition that serves as a phase counterpart to the small gain theorem \cite{wang2025algebraic}, offering a powerful tool for analyzing the stability and robustness of multilinear control systems.  This result effectively completes a dual set of tools---gain and phase---for multilinear control system analysis.
\end{enumerate}

The rest of this paper is organized as follows. Section~\ref{Preliminary} introduces the fundamental notation and some preliminary results of tensors defined via the Einstein product. Section~\ref{Characterization} presents the sectorial tensor decomposition and formally defines tensor phases. Section~\ref{Phases of compression} studies the maximin and
minimax properties of tensor phases and their behavior under compression. Section~\ref{Compound numerical ranges} explores the relationship between the compound spectrum and  the compound numerical ranges. Section~\ref{Phases of tensor product and sum} analyzes the phase bounds for tensor products and sums. Section \ref{Rank robustness against perturbations} studies rank robustness of product of tensors in terms of tensor phases. Section~\ref{Small Phase Theorem} presents and proves the tensor version of the small phase theorem. Section \ref{sec:quasi-semi} includes results for quasi-sectorial and semi-sectorial tensors.  Finally, Section~\ref{conclusion} concludes the paper and discusses future research directions.
 
\section{Preliminaries}\label{Preliminary}

In this section, some notions of tensors are collected. Throughout this paper, $\imath=\sqrt{-1}$ denotes the imaginary unit. $\mathbb{R}$ is the field of real numbers,  $\mathbb{C}$ is the field of complex numbers. Calligraphic letters $\mathcal{A}, \mathcal{B},\cdots$ are used to represent tensors. For a tensor $\mathcal{A} \in \mathbb{C}^{(I_{1} \times \cdots \times I_{M})\times (K_{1} \times \cdots \times K_{N})}$, its order is $M+N$,  and its dimensions are separated into two parts: $(I_{1},\cdots, I_{M})$ and $(K_{1}, \cdots, K_{N})$. Specifically, the dimension of its $i$th row is $I_i$ and that of its $k$th column is $K_k$. In particular, if $N=M$ and $I_{1} =J_{1},\ldots, I_{N} =J_{N}$, the tensor is called an even-order square tensor.  By  writing $\mathcal{X} \in \mathbb{C}^{I_{1} \times \cdots \times I_{N}}$ we mean either $\mathcal{X} \in \mathbb{C}^{(I_{1} \times \cdots \times I_{N})\times(1)}$ (namely, no columns) or $\mathcal{X} \in \mathbb{C}^{(1)\times(I_{1} \times \cdots \times I_{N})}$ (namely, no rows), and which one of them is used should be clear from the context.  For convenience, denote $|\textbf{I}|=\prod_{n=1}^{N}I_{n}$, and $[n]= \{1,2,\ldots,n\}$. A nonzero scalar $a\in \mathbb{C}$ can be represented in the polar form as $a = \sigma e^{\imath \phi}$, where $\sigma>0$ is the magnitude and $\phi$ is  the phase (argument). In this paper, we restrict   $\phi \in(-\pi,\pi]$ and denote it by $\angle a $. Next, we introduce the Einstein product for tensors.

\begin{definition} [\cite{einstein1916foundation, HQ18,chen2019multilinear,r12}]
Given two tensors $\mathcal{A} \in \mathbb{C}^{(I_{1} \times \cdots \times I_{M})\times (K_{1} \times \cdots \times K_{N})}$ and $\mathcal{B} \in \mathbb{C}^{(K_{1} \times \cdots \times K_{N})\times (J_{1} \times \cdots \times J_{L})}$, the Einstein product  $\mathcal{A} *_{N} \mathcal{B} \in \mathbb{C}^{(I_{1} \times \cdots \times I_{M})\times (J_{1} \times \cdots \times J_{L})}$ is defined  element-wise via
\begin{equation*}
(\mathcal{A} *_{N}\mathcal{B})_{i_{1}\cdots i_{M}j_{1}\cdots j_{L}} =
\sum_{k_{1},\cdots,k_{N}} a_{i_{1}\cdots i_{M}k_{1}\cdots k_{N}} b_{k_{1}\cdots k_{N}j_{1}\cdots j_{L}}.
\end{equation*}
\end{definition}


The following are some elementary tensor operations, which are natural generalizations of their matrix counterparts.

\begin{definition}
Given a tensor $\mathcal{A} \in \mathbb{C}^{(I_{1} \times \cdots \times I_{N})\times (J_{1} \times \cdots \times J_{M})}$ and a complex number $c$,  we define their scalar multiplication element-wise  as 
\begin{equation*}
(c \mathcal{A})_{i_{1}\cdots i_{N}j_{1}\cdots j_{M}}=c a_{i_{1}\cdots i_{N}j_{1}\cdots j_{M}}.
\end{equation*}
The conjugate transpose of the tensor $\A$  is defined element-wise as
\begin{equation*}
( \mathcal{A}^{H})_{j_{1}\cdots j_{M}i_{1}\cdots i_{N}}= \bar{a}_{i_{1}\cdots i_{N}j_{1}\cdots j_{M}}.
\end{equation*}
A tensor $\A$ is said to be \emph{Hermitian} if $\A = \A^{H}$.
\end{definition}

\begin{definition} \label{def:diagonal}
A tensor $\mathcal{D} \in \mathbb{C}^{(I_{1} \times \cdots \times I_{N})\times (J_{1} \times \cdots \times J_{N})}$ is  a diagonal tensor if $d_{i_{1}\cdots i_{N}j_{1}\cdots j_{N}}=0$, whenever $(i_{1}, \cdots,  i_{N}) \neq (j_{1}, \cdots,  j_{N})$.
\end{definition}
Notice that a diagonal tensor defined above is even-order, but may not be square.

 \begin{definition}
     An even-order square tensor $ \mathcal{I} \in \mathbb{C}^{(I_{1} \times \cdots \times I_{N})\times (I_{1} \times \cdots \times I_{N})}$ is called an identity tensor if it is a diagonal tensor with diagonal entries $\mathcal{I}_{i_{1}\cdots i_{N}i_{1}\cdots i_{N}}=1$.
 \end{definition} 
 
\begin{definition}
Given $\mathcal{A} \in \mathbb{C}^{(I_{1} \times \cdots \times I_{N})\times (I_{1} \times \cdots \times I_{N})}$ an even-order square tensor, if there exists $\mathcal{B} \in \mathbb{C}^{(I_{1} \times \cdots \times I_{N})\times (I_{1} \times \cdots \times I_{N})}$ such that
\begin{equation*}
\mathcal{A} *_{N} \mathcal{B}=\mathcal{B} *_{N} \mathcal{A}=\mathcal{I},
\end{equation*}
then  $\mathcal{B}$ is called the inverse of $\mathcal{A}$, denoted   $\mathcal{A}^{-1}$.
\end{definition}

\begin{definition}[\cite{chen2021multilinear}]
For a square tensor  $\mathcal{A} \in \mathbb{C}^{(I_{1} \times \cdots \times I_{N})\times (I_{1} \times \cdots \times I_{N})}$, if a complex number $\lambda$ and a non-zero tensor $\mathcal{X} \in \mathbb{C}^{I_{1} \times \cdots \times I_{N}}$ satisfy $\mathcal{A} *_{N} \mathcal{X}=\lambda \mathcal{X}$,
then we say that $\lambda$ is an eigenvalue of $\mathcal{A}$, and $\mathcal{X}$ is the corresponding eigentensor.
\end{definition}

\begin{definition}[\cite{r6}]
Given a square tensor $\mathcal{A} \in \mathbb{C}^{(I_{1} \times \cdots \times I_{N})\times (I_{1} \times \cdots \times I_{N})}$  with eigenvalues $\lambda_{1}, \lambda_{2}, \ldots, \lambda_{|\textbf{I}|}$, its  determinant  is defined as $det(\mathcal{A})=\prod_{i=1}^{|\textbf{I}|}\lambda_{i}$. 
\end{definition}

\begin{definition}[\cite{MR3914335}] \label{def:rank}
    The rank of a tensor $\mathcal{A} \in \mathbb{C}^{(I_{1} \times \cdots \times I_{M})\times (K_{1} \times \cdots \times K_{N})}$, denoted  $rank(\mathcal{A})$, is defined to be the number of non-zero eigenvalues of $\mathcal{A}^{H}*_{N} \mathcal{A}$. If $rank(\mathcal{A})=\min\{|\textbf{I}|, |\textbf{K}|\}$,  $\mathcal{A}$ is called a \emph{nonsingular} tensor. Clearly, if $\mathcal{A} $ is square, i.e., $N=M$ and $K_1 = I_1, \ldots, K_M=I_M$, then $rank(\mathcal{A})$ equals the number of the non-zero eigenvalues of $\mathcal{A}$ and hence $\mathcal{A}$ is nonsingular if and only if all its eigenvalues are non-zero. 
\end{definition}

Given tensors $\mathcal{X}, \mathcal{Y} \in \mathbb{C}^{I_{1} \times \cdots \times I_{N}}$, an inner product $\langle  \cdot,\cdot\rangle $ is defined as $\langle \mathcal{X} , \mathcal{Y}\rangle =\mathcal{Y}^{H} *_{N} \mathcal{X}$. The
Frobenius norm induced by this inner product is  $||\mathcal{X}||_{F}=\sqrt{\langle  \mathcal{X} , \mathcal{X}\rangle }$. A tensor $\mathcal{A}$  is called a unit tensor if  $||\mathcal{A}||_{F}=1$.

\begin{definition}
An even-order square tensor $\mathcal{A} \in \mathbb{C}^{(I_{1} \times \cdots \times I_{N})\times (I_{1} \times \cdots \times I_{N})}$ is called  a unitary tensor if $\mathcal{A}^{H}=\mathcal{A}^{-1}$, and  a positive-definite tensor if $\langle  \mathcal{A} *_{N} \mathcal{X} , \mathcal{X} \rangle $ is positive for all $\mathcal{O}\neq\mathcal{X} \in \mathbb{C}^{I_{1} \times \cdots \times I_{N}}$.
\end{definition}

\section{Phase theory of tensors} 
\subsection{Phases of sectorial tensors} \label{Characterization}

In this subsection,  we first introduce the numerical range for  even-order square tensors. After that we present a sectorial tensor decomposition, which allows us to define phases for sectorial tensors. 

\begin{definition}[\cite{r1}]
Given an even-order square tensor $\mathcal{A} \in \mathbb{C}^{(I_{1} \times \cdots \times I_{N})\times (I_{1} \times \cdots \times I_{N})}$, its numerical range, denoted $W(\mathcal{A})$, is defined as
\begin{multline*}
W(\mathcal{A}) = \{ \langle \mathcal{A} *_{N} \mathcal{X}, \mathcal{X} \rangle : \\
\mathcal{X} \text{ is a unit tensor in } \mathbb{C}^{I_{1} \times \cdots \times I_{N}} \}.
\end{multline*}
 The \emph{field angle} of $ \mathcal{A}$, denoted $\delta(\mathcal{A})$, is the angle subtended by the two supporting rays of $ W(\mathcal{A})$ originating from the origin.
\end{definition}

\begin{definition}
Given $\mathcal{A} \in \mathbb{C}^{(I_{1} \times \cdots \times I_{N})\times (I_{1} \times \cdots \times I_{N})}$, its angular numerical range, denoted $W' (\mathcal{A})$, is defined as
\begin{multline*}
W'(\mathcal{A}) = \{ \langle \mathcal{A} *_{N} \mathcal{X}, \mathcal{X} \rangle : \\
\mathcal{X} \text{ is a nonzero tensor in } \mathbb{C}^{I_{1} \times \cdots \times I_{N}} \}.
\end{multline*}
\end{definition}

Clearly, $ W(\mathcal{A}) \subseteq  W' (\mathcal{A})$. But they have the same field angle.

\begin{lemma}[\cite{r1}] \label{convex}
 The numerical range  of an even-order square tensor is convex.
\end{lemma}


Next, we define sectorial tensors as the natural generalization of the matrix case \cite{r2}.
 
\begin{definition} \label{def:sectorial}
An even-order square  tensor  $\mathcal{A}$ is called sectorial if $0 \notin W(\mathcal{A})$.
\end{definition}

According to Lemma \ref{convex},   the numerical range $W(\mathcal{A})$ of a sectorial tensor $\A$ is a  convex set  contained in an open half plane.

If an even-order square tensor $\mathcal{D}$ is diagonal, then its numerical range is of the form
\begin{multline*} \label{W_D}
W(\mathcal{D})= \left\{ \sum_{i_{1},i_{2},\cdots,i_{N}}d_{i_{1}i_{2}\cdots i_{N}i_{1}i_{2}\cdots i_{N}}|x_{i_{1}i_{2}\cdots i_{N}}|^{2}: \right. \\
\left. \sum_{i_{1},i_{2},\cdots,i_{N}}|x_{i_{1}i_{2}\cdots i_{N}}|^{2}=1 \right\}.
\end{multline*}
Define a set $P(\mathcal{D})=\{d_{i_{1}i_{2}\cdots i_{N}i_{1}i_{2}\cdots i_{N}}: \forall i_1\in [I_1], \ldots, \forall i_N\in [I_N] \}$. Due to the arbitrariness of $\mathcal{X}$,  $P(\mathcal{D})$ is actually the set of all eigenvalues of the diagonal tensor $\mathcal{D}$. Clearly, $W(\mathcal{D})= conv
(P(\mathcal{D}))$, where $conv(P(\mathcal{D}))$ represents the convex hull of $P(\mathcal{D})$. Therefore, an even-order square diagonal tensor $\mathcal{D}$ is  sectorial if and only if $P(\mathcal{D})$ is in an open half plane.

Like Hermitian matrices, Hermitian tensors also have   spectral decompositions, as given below. 

\begin{lemma}[\cite{erfanifar2024polar}]\label{th2_1}
Given a Hermitian tensor $\mathcal{A}\in \mathbb{C}^{(I_{1} \times \cdots \times I_{N})\times (I_{1} \times \cdots \times I_{N})}$, there exists a unitary tensor $\mathcal{U}$ of the same size such that
\begin{equation*}
\mathcal{U}^{H} *_{N} \mathcal{A} *_{N} \mathcal{U}= \mathcal{D},
\end{equation*}
where $\mathcal{D}$ is a diagonal tensor containing all eigenvalues of $\mathcal{A}$.
\end{lemma}

Similar to the matrix case, a positive-definite tensor and a Hermitian tensor is simultaneously  congruent to diagonal tensors, as given by the following result, which will be used in the proof of the sectorial tensor  decomposition  Theorem \ref{th2_3}. 

\begin{lemma}\label{th2_2}
Let $\mathcal{A},\mathcal{B}\in \mathbb{C}^{(I_{1} \times \cdots \times I_{N})\times (I_{1} \times \cdots \times I_{N})}$ be Hermitian tensors with  $\mathcal{A}$ being positive-definite. Then there exists a nonsingular tensor $\mathcal{C}\in \mathbb{C}^{(I_{1} \times \cdots \times I_{N})\times (I_{1} \times \cdots \times I_{N})}$, such that
\begin{equation}\label{eq:ABCD}
\mathcal{C}^{H} *_{N} \mathcal{A} *_{N} \mathcal{C}= \mathcal{I}, \quad \mathcal{C}^{H} *_{N} \mathcal{B} *_{N} \mathcal{C}= \mathcal{D},
\end{equation}
where $\mathcal{I}$ is the identity tensor and $\mathcal{D}$ is a real diagonal tensor.
\end{lemma}
\begin{proof}
For the positive-definite tensor $\mathcal{A}$, according to Ref. \rm \cite{erfanifar2024polar} there exists a nonsingular tensor  $ \mathcal{P}$ such that $\mathcal{P}^{H} *_{N} \mathcal{A} *_{N} \mathcal{P}= \mathcal{I}$.  Consider the Hermitian tensor $\mathcal{P}^{H} *_{N} \mathcal{B} *_{N} \mathcal{P}$. By Lemma  \ref{th2_1}   there exists  a unitary tensor $\mathcal{Q}$ such that $\mathcal{Q}^{H} *_{N} (\mathcal{P}^{H} *_{N} \mathcal{B} *_{N} \mathcal{P}) *_{N} \mathcal{Q}= \mathcal{D}$, where $\mathcal{D}$ is a diagonal tensor.  
Moreover, as $\B$ is Hermitian, $\D$ is a real tensor. Let $\mathcal{C}=\mathcal{P} *_{N} \mathcal{Q}$. Then $\mathcal{C}$ is the constructed tensor that yields Eq. \eqref{eq:ABCD}.
\end{proof}

Using Lemmas \ref{th2_1} and \ref{th2_2}, we can derive the following tensor decomposition theorem, whose matrix counterpart can be found in Refs. \cite{r9,r3,r4}.

\begin{theorem}\label{th2_3}
Let $\mathcal{A}\in \mathbb{C}^{(I_{1} \times \cdots \times I_{N})\times (I_{1} \times \cdots \times I_{N})}$ be a sectorial tensor.  There exist a non-singular tensor $\mathcal{Q}$ and a unitary diagonal  tensor $\mathcal{D} $ of the same size, such that
\begin{equation}\label{eq:STD}
\mathcal{A}= \mathcal{Q}^{H} *_{N} \mathcal{D} *_{N} \mathcal{Q}.
\end{equation}
\end{theorem}
\begin{proof}
Define a set of angles  \begin{multline*}
S_{W}(\mathcal{A})= \{ \angle \langle \mathcal{A} *_{N} \mathcal{X}, \mathcal{X} \rangle : \\
\mathcal{X} \text{ is a nonzero tensor in } \mathbb{C}^{I_{1} \times \cdots \times I_{N}} \}.
\end{multline*} 
By Lemma \ref{convex}, the sectorialness of $\mathcal{A}$ implies that its numerical range $W(\mathcal{A})$ is contained in an open half plane, so there exists some $\theta$ such that $S_{W}(\mathcal{A}) \subset (\theta, \theta+\pi)$. Hence, $S_{W}(e^{-\imath \theta}\mathcal{A}) \subset (0, \pi)$. Without loss of generality, we  assume $S_{W}(\mathcal{A}) \subset (0, \pi)$. Defining the Hermitian tensors $\mathcal{H} = \frac{\mathcal{A} + \mathcal{A}^{H}}{2}$ and $\mathcal{K} = \frac{\mathcal{A} - \mathcal{A}^{H}}{2\imath}$, we have the decomposition $\mathcal{A} = \mathcal{H} + \imath \mathcal{K}$. Then for any nonzero tensor $\mathcal{X}$,  $\langle \mathcal{A} *_{N} \mathcal{X}, \mathcal{X} \rangle=\langle \mathcal{H} *_{N} \mathcal{X}, \mathcal{X} \rangle + \imath \langle \mathcal{K} *_{N} \mathcal{X}, \mathcal{X} \rangle$, where  $\langle \mathcal{H} *_{N} \mathcal{X}, \mathcal{X} \rangle$ is the real part and  $\langle \mathcal{K} *_{N} \mathcal{X}, \mathcal{X} \rangle$ is the imaginary part. Since $S_{W}(\mathcal{A}) \subset (0, \pi)$, it follows that $\langle \mathcal{K} *_{N} \mathcal{X}, \mathcal{X} \rangle > 0$ for all $\mathcal{X} \neq 0$, which establishes that $\mathcal{K}$ is positive definite. Consequently, by Lemma \ref{th2_2}, there exists a nonsingular tensor $\mathcal{C}\in \mathbb{C}^{(I_{1} \times \cdots \times I_{N})\times (I_{1} \times \cdots \times I_{N})}$ such that
\begin{equation*}
\mathcal{C}^{H} *_{N} \mathcal{H} *_{N} \mathcal{C}= \mathcal{D}_{0}, \quad \mathcal{C}^{H} *_{N} \mathcal{K} *_{N} \mathcal{C}= \mathcal{I},
\end{equation*}
where $\mathcal{D}_0$ is real and diagonal. 
Let $\mathcal{D}_{1}=\mathcal{D}_{0}+\imath\mathcal{I}$.  Then $\mathcal{D}_{1}$ is nonsingular, and hence $\mathcal{D}_{1}^{H}*_{N}\mathcal{D}_{1}$ is real  and  positive-definite. Denote 
\begin{align*}
   \mathcal{D}&=\mathcal{D}_{1}*_{N}(\mathcal{D}_{1}^{H}*_{N}\mathcal{D}_{1})^{-\frac{1}{2}}\\
   &=(\mathcal{D}_{1}^{H}*_{N}\mathcal{D}_{1})^{-\frac{1}{4}}*_{N}\mathcal{D}_{1}*_{N}(\mathcal{D}_{1}^{H}*_{N}\mathcal{D}_{1})^{-\frac{1}{4}} 
\end{align*} 
and $\mathcal{Q}=(\mathcal{D}_{1}^{H}*_{N}\mathcal{D}_{1})^{\frac{1}{4}}*_{N}\mathcal{C}^{-1}$. Then $\mathcal{D}$ is unitary and diagonal, and
\begin{eqnarray*}
\mathcal{A}&=&\mathcal{H}+\imath\mathcal{K}=\mathcal{Q}^{H} *_{N} \mathcal{D}_{0} *_{N} \mathcal{Q}+\imath(\mathcal{Q}^{H} *_{N} \mathcal{I} *_{N} \mathcal{Q}) \\
&=& \mathcal{Q}^{H} *_{N} \mathcal{D} *_{N} \mathcal{Q},
\end{eqnarray*}
which is Eq. \eqref{eq:STD}. The proof is completed.
\end{proof}

In this paper, the tensor  decomposition in Theorem \ref{th2_3} is referred to as the \emph{sectorial tensor  decomposition}. A sectorial tensor decomposition for a sectorial tensor is not unique. Nevertheless, the diagonal unitary tensor $\mathcal{D}$ is unique up to a permutation, whose matrix counterpart has been pointed out in \cite{r9}.

\begin{theorem} \label{thm:unique}
    The diagonal unitary tensor $\mathcal{D}$ in Theorem \ref{th2_3} is unique up to a permutation.
\end{theorem}
\begin{proof}
    Suppose $\mathcal{A}= \mathcal{Q}_{1}^{H} *_{N} \mathcal{D}_{1} *_{N} \mathcal{Q}_{1}= \mathcal{Q}_{2}^{H} *_{N} \mathcal{D}_{2} *_{N} \mathcal{Q}_{2}$ are two sectorial tensor decompositions of a sectorial tensor $\mathcal{A}$. Simple algebraic manipulations yield
$$\mathcal{D}_{2}=\mathcal{T}^{H}*_{N}\mathcal{D}_{1}*_{N}\mathcal{T},$$
    where $\mathcal{T}=\mathcal{Q}_{1}*_{N}\mathcal{Q}_{2}^{-1}$ is a nonsingular tensor,  $\mathcal{D}_{1}$ and $\mathcal{D}_{2}$ are unitary diagonal tensors. Denote $\mathcal{D}_{1}=\mathcal{C}_{1}+\imath\mathcal{S}_{1},\mathcal{D}_{2}=\mathcal{C}_{2}+\imath\mathcal{S}_{2}$, where $\mathcal{C}_{1},\mathcal{C}_{2},\mathcal{S}_{1},\mathcal{S}_{2}$ are all real diagonal tensors. We can obtain $$\mathcal{C}_{2}=\mathcal{T}^{H}*_{N}\mathcal{C}_{1}*_{N}\mathcal{T},\quad\mathcal{S}_{2}=\mathcal{T}^{H}*_{N}\mathcal{S}_{1}*_{N}\mathcal{T}.$$
    If there are $k$ elements $1$ in the diagonal part of $\mathcal{D}_{1}$, then $rank(\mathcal{S}_{1})=n-k$. Due to the reversibility of $\mathcal{T}$, we know $rank(\mathcal{S}_{2})=n-k$, which indicate that the number  $1$ in the diagonal part of $\mathcal{D}_{2}$ is also $k$. For an element $\beta$ with $|\beta|=1$ in the diagonal part of $\mathcal{D}_{1}$.
    Consider $e^{-\imath\angle\beta}\mathcal{D}_{1}$ and $e^{-\imath\angle\beta}\mathcal{D}_{2}$, we also have the decompositions
    $$(e^{-\imath\angle\beta}\mathcal{C}_{2})=\mathcal{T}^{H}*_{N}(e^{-\imath\angle\beta}\mathcal{C}_{1})*_{N}\mathcal{T},$$
    $$(e^{-\imath\angle\beta}\mathcal{S}_{2})=\mathcal{T}^{H}*_{N}(e^{-\imath\angle\beta}\mathcal{S}_{1})*_{N}\mathcal{T}.$$
    Following the same procedure, it can be concluded that  the diagonal part of $\mathcal{D}_{1}$ contains the same number of elements $\beta$ with the diagonal part of $\mathcal{D}_{2}$, which complete the proof. 
\end{proof}

By Theorem \ref{th2_3}, we have
\begin{align*}
W' (\mathcal{A}) & = \{ \langle \mathcal{A} *_{N} \mathcal{X}, \mathcal{X} \rangle: \nonumber \\
& \quad \mathcal{X} \text{ is a nonzero tensor in } \mathbb{C}^{I_{1} \times \cdots \times I_{N}} \} \nonumber \\
& = \{ \langle \mathcal{D} *_{N} (\mathcal{Q} *_{N} \mathcal{X}), (\mathcal{Q} *_{N} \mathcal{X}) \rangle : \nonumber \\
& \quad \mathcal{X} \text{ is a nonzero tensor in } \mathbb{C}^{I_{1} \times \cdots \times I_{N}} \} \nonumber \\
& = W' (\mathcal{D}).
\label{eq:W_A_D}
\end{align*}
 Therefore, if $\mathcal{A}$ is a sectorial tensor, then the diagonal unitary tensor  $\mathcal{D}$ in the sectorial tensor  decomposition is sectorial too.

 We are ready to define  phases of sectorial tensors.
\begin{definition}\label{def:phases}
Given the  sectorial tensor  decomposition \eqref{eq:STD} of a sectorial tensor $\mathcal{A}$, its phases are defined as the phases of the eigenvalues of the diagonal unitary tensor $\mathcal{D}$. We order the phases by
\begin{equation*}
\bar{\Phi}(\mathcal{A}) = \Phi_{1}(\mathcal{A}) \geq   \Phi_{2}(\mathcal{A}) \geq \cdots \geq \Phi_{|\textbf{I}|}(\mathcal{A})
 = \underline{\Phi}(\mathcal{A}).
\end{equation*}
\end{definition}

 As $\A$ is sectorial,   by Definition \ref{def:sectorial},  $\bar{\Phi}(\mathcal{A})-\underline{\Phi}(\mathcal{A})< \pi$. 


 As the phases of a tensor are defined via its sectorial tensor  decomposition which is a congruent transformation, it is important to show that  tensor phases are invariant under congruent transformations.
\begin{lemma}\label{th2_4}
The phases of a sectorial tensor $\mathcal{A}$ are invariant under congruent transformations, i.e.,  $\Phi(\mathcal{A})=\Phi(\mathcal{Q}^{H} *_{N} \mathcal{A} *_{N} \mathcal{Q}) $ for an arbitrary nonsingular tensor $\mathcal{Q}$.
\end{lemma}

 Lemma \ref{th2_4} is an immediate consequence of Theorem \ref{thm:unique}.

The following result shows that tensor eigenvalues  are invariant under the similarity transformations.

\begin{lemma}\label{th2_5}
Similarity transformations under the Einstein product do not change the eigenvalues of tensors. i.e., if $\mathcal{T}$ is a nonsingular tensor, then $\mathcal{A}$ and $\mathcal{T}^{-1} *_{N} \mathcal{A}  *_{N} \mathcal{T}$ have the same eigenvalues.
\end{lemma}

The proof of Lemma \ref{th2_5} is straightforward, thus is omitted.



In general, given a sectorial tensor $\A$, it is not easy to perform the sectorial tensor transformation in Theorem \ref{th2_3} to find the unitary diagonal tensor $\D$ to get the phases of $\A$; \cite{r1, ke2016numerical}.  Fortunately, Lemma \ref{th2_5} provide an alternative way.  Let $\mathcal{A}\in \mathbb{C}^{(I_{1} \times \cdots \times I_{N})\times (I_{1} \times \cdots \times I_{N})}$ be a sectorial tensor with the sectorial tensor decomposition $\mathcal{A}= \mathcal{T}^{H} *_{N} \mathcal{D} *_{N} \mathcal{T}$. Then 
\begin{equation*}\label{eq3_4}
	\mathcal{A}^{-1} *_{N} \mathcal{A}^{H}=\mathcal{T}^{-1} *_{N} \mathcal{D}^{-1} *_{N} \mathcal{D}^{H} *_{N} \mathcal{T}.
\end{equation*}
Because $\mathcal{D}$ is a diagonal unitary tensor, $\mathcal{D}^{-1} *_{N} \mathcal{D}^{H}=\mathcal{D}^{-2}$. Therefore, $\mathcal{A}^{-1} *_{N} \mathcal{A}^{H}=\mathcal{T}^{-1} *_{N} \mathcal{D}^{-2}  *_{N} \mathcal{T}$. In other words,  $\mathcal{A}^{-1} *_{N} \mathcal{A}^{H}$ is similar to the diagonal unitary tensor $\mathcal{D}^{-2}$. Therefore, if we want to compute the phases of $\mathcal{A}$, by Lemma \ref{th2_5} we can  calculate the eigenvalues of $\mathcal{A}^{-1} *_{N} \mathcal{A}^{H}$. See  Example \ref{calphase} below for a simple illustration.

\begin{algorithm}
\caption{Compute the phases of tensor $\mathcal{A}$}
\label{alg}
\begin{algorithmic}[1]
\Require Tensor $\mathcal{A}$.
\Ensure All phases of $\mathcal{A}$.
\State Compute tensor $\mathcal{A}^{-1} *_{N} \mathcal{A}^{H}$.
\State Compute all the eigenvalues of $\mathcal{A}^{-1} *_{N} \mathcal{A}^{H}$ \cite{cui2016eigenvalue}, and denote them by $\lambda_{1}, \lambda_{2},\cdots,\lambda_{|\textbf{I}|}$.
\State Record the phases of $\lambda_{1}, \lambda_{2},\cdots,\lambda_{|\textbf{I}|}$ as $\theta_{1},\theta_{2},\cdots,\theta_{|\textbf{I}|}$.
\State The phases of $\mathcal{A}$ are $-\frac{1}{2}\theta_{1},-\frac{1}{2}\theta_{2},\cdots,-\frac{1}{2}\theta_{|\textbf{I}|}$.
\end{algorithmic}
\end{algorithm}

\bigskip

\begin{example}\label{calphase}
    Consider a tensor 
    $\mathcal{A}\in\mathbb{C}^{(2\times 2)\times (2\times 2)}$ with
    $$\mathcal{A}(1,1,:,:)=\begin{pmatrix}
        e^{\imath\theta_{1}}&e^{\imath\theta_{1}}\\
        e^{\imath\theta_{1}}&0
    \end{pmatrix},$$
    $$ 
    \mathcal{A}(1,2,:,:)=\begin{pmatrix}
        e^{\imath\theta_{1}}&e^{\imath\theta_{1}}+e^{\imath\theta_{2}}\\
        e^{\imath\theta_{1}}&e^{\imath\theta_{2}}
    \end{pmatrix},
    $$
    $$\mathcal{A}(2,1,:,:)=\begin{pmatrix}
        e^{\imath\theta_{1}}&e^{\imath\theta_{1}}\\
        e^{\imath\theta_{1}}+e^{\imath\theta_{3}}&0
    \end{pmatrix},$$
    $$\mathcal{A}(2,2,:,:)=\begin{pmatrix}
        0&e^{\imath\theta_{2}}\\
        0&e^{\imath\theta_{4}}+e^{\imath\theta_{2}}
    \end{pmatrix},
    $$
    where $\theta_{1},\theta_{2},\theta_{3},\theta_{4}\in(-\pi,\pi)$. 
    By Algorithm \ref{alg}, we have
    $$(\mathcal{A}^{-1} *_{2} \mathcal{A}^{H})(1,1,:,:)=\begin{pmatrix}
        e^{-2\imath\theta_{1}}&e^{-2\imath\theta_{1}}-e^{-2\imath\theta_{2}}\\
        e^{-2\imath\theta_{1}}-e^{-2\imath\theta_{3}}&-e^{-2\imath\theta_{2}}+e^{-2\imath\theta_{4}}
    \end{pmatrix},$$
    $$(\mathcal{A}^{-1} *_{2} \mathcal{A}^{H})(1,2,:,:)=\begin{pmatrix}
        0&e^{-2\imath\theta_{2}}\\
        0&e^{-2\imath\theta_{2}}-e^{-2\imath\theta_{4}}
    \end{pmatrix},
    $$
    $$(\mathcal{A}^{-1} *_{2} \mathcal{A}^{H})(2,1,:,:)=\begin{pmatrix}
        0&0\\
        e^{-2\imath\theta_{3}}&0
    \end{pmatrix},$$
    $$(\mathcal{A}^{-1} *_{2} \mathcal{A}^{H})(2,2,:,:)=\begin{pmatrix}
        0&0\\
        0&e^{-2\imath\theta_{4}}
    \end{pmatrix}.
    $$
Eigenvalues of $\mathcal{A}^{-1} *_{2} \mathcal{A}^{H}$ are $e^{-2\imath\theta_{1}},e^{-2\imath\theta_{2}},e^{-2\imath\theta_{3}},e^{-2\imath\theta_{4}}$, and their phases  are $-2\theta_{1}, -2\theta_{2}, -2\theta_{3}, -2\theta_{4}$. Multiply all of these by $-\frac{1}{2}$ and we obtain all phases of $\mathcal{A}$ as $\theta_{1},\theta_{2},\theta_{3},\theta_{4}$.
\end{example}

\subsection{Phases of compressions of sectorial tensors}\label{Phases of compression}
In this subsection, we first present the maximin and minimax expressions of tensor phases, after that we introduce the inequality between the phases of a sectorial tensor and those of its compressions. 
The matrix case of the following result is proved in \cite{r3}.

\begin{lemma}\label{th3_1}
	The phases of a sectorial tensor $\mathcal{A}\in \mathbb{C}^{(I_{1} \times \cdots \times I_{N})\times (I_{1} \times \cdots \times I_{N})}$  enjoy the following properties.
\begin{align}
    \Phi_{i}(\mathcal{A})
    &= \max_{\substack{\mathcal{M}: \dim \mathcal{M}=i}} \min_{\substack{\mathcal{X} \in \mathcal{M}, \\ \|\mathcal{X}\|=1}}  \angle (\mathcal{X}^{H} *_{N} \mathcal{A} *_{N} \mathcal{X}) \nonumber  \\
    &= \min_{\substack{\mathcal{N}: \dim \mathcal{N}=\\
    |\textbf{I}|-i+1}} \max_{\substack{\mathcal{X} \in \mathcal{N}, \\ \|\mathcal{X}\|=1}}  \angle (\mathcal{X}^{H} *_{N} \mathcal{A} *_{N} \mathcal{X}), \label{eq:dec5_a} 
\end{align}
where $\mathcal{M},\mathcal{N}$ are the subspaces of $ \mathbb{C}^{I_{1} \times \cdots \times I_{N}}$.  In particular,
\begin{equation*}
    \bar{\Phi}(\mathcal{A})
    =\max_{\substack{\mathcal{X} \in \mathbb{C}^{I_{1} \times \cdots \times I_{N}} \\||\mathcal{X}||=1}}  \angle (\mathcal{X}^{H} *_{N} \mathcal{A}  *_{N} \mathcal{X} ) ,
\end{equation*}
and
\begin{equation*}
    \underline{\Phi}(\mathcal{A})
    =\min_{\substack{\mathcal{X} \in \mathbb{C}^{I_{1} \times \cdots \times I_{N}} \\||\mathcal{X}||=1} } \angle (\mathcal{X}^{H} *_{N} \mathcal{A}  *_{N} \mathcal{X} ).
\end{equation*}
\end{lemma}
\begin{proof}
   We only prove the first half of Eq. \eqref{eq:dec5_a}, and the second half can be proven similarly. Considering the sectorial tensor  decomposition $\mathcal{A}= \mathcal{T}^{H} *_{N} \mathcal{D} *_{N} \mathcal{T}$ in Theorem \ref{th2_3}. Notice that if $\mathcal{X}$ takes an element from an $i$-dimensional subspace, then $\mathcal{T} *_{N}\mathcal{X}$  also takes an element  from an $i$-dimensional subspace. Let $\mathcal{Y} = \frac{\mathcal{T} *_{N}\mathcal{X}}{||\mathcal{T} *_{N}\mathcal{X}||}$. Then we have
   \begin{multline}\label{eq3_1}
    \max_{\mathcal{M}:\dim \mathcal{M}=i} \min_{\substack{\mathcal{X} \in \mathcal{M},\\ \|\mathcal{X}\|=1}}  
    \angle (\mathcal{X}^{H} *_{N} \mathcal{A} *_{N} \mathcal{X})  \\
    =\max_{\mathcal{M}:\dim \mathcal{M}=i} \min_{\substack{\mathcal{Y} \in \mathcal{M},\\ \|\mathcal{Y}\|=1}}  
    \angle (\mathcal{Y}^{H} *_{N} \mathcal{D} *_{N} \mathcal{Y}).
\end{multline}
    Because $\mathcal{D}$ is a diagonal unitary tensor,  the right-hand side of Eq. (\ref{eq3_1}) indicates that the solution to its left-hand side is the $i$-th largest phase of $\mathcal{A}$, namely $\Phi_{i}(\mathcal{A})$, thus establishing the first half of Eq. \eqref{eq:dec5_a}.
\end{proof}

Lemma \ref{th3_1} will be used in Section \ref{Phases of tensor product and sum} for studying  phases of sum  and product of sectorial tensors.

Next, we define compressions of tensors.

\begin{definition} \label{def:compression}
Given a tensor $\mathcal{A}\in \mathbb{C}^{(I_{1} \times \cdots \times I_{N})\times (I_{1} \times \cdots \times I_{N})}$, let $\mathcal{U} \in \mathbb{C}^{(I_{1} \times \cdots \times I_{N})\times (J_{1} \times \cdots \times J_{L})}$  with $|\textbf{J}|<|\textbf{I}|$ be a  column orthogonal tensor, that is,   $\mathcal{U}^{H} *_{N} \mathcal{U}=\mathcal{I}$.  The tensor 
$\tilde{\mathcal{A}}=\mathcal{U}^{H} *_{N} \mathcal{A} *_{N}\mathcal{U} \in  \mathbb{C}^{(J_1 \times \cdots \times J_L)\times (J_{1} \times \cdots \times J_{L})}$ is called a \emph{compression} of $\mathcal{A}$. 
\end{definition}

By construction, the size of a  compression $\tilde{\mathcal{A}}$ is smaller than the size of the original tensor  $\mathcal{A}$, this might be useful for efficient data processing. But we need to quantify the level of approximation via  compression. The following result gives the relation between  the phases of $\mathcal{A}$ and  its compression $\tilde{\mathcal{A}}$, which extends the matrix case in \cite[Lemma 7]{r5}. 

\begin{theorem}\label{th3_2}
Let $\mathcal{A}\in \mathbb{C}^{(I_{1} \times \cdots \times I_{N})\times (I_{1} \times \cdots \times I_{N})}$ be sectorial and $\tilde{\mathcal{A}}=\mathcal{U}^{H} *_{N} \mathcal{A} *_{N}\mathcal{U}$ be a compression of $\mathcal{A}$. Then $\tilde{\mathcal{A}}$ is also sectorial and its phases satisfy
\begin{equation} \label{eq:nov26_1}
    \Phi_{i}(\mathcal{A}) \geq \Phi_{i}(\tilde{\mathcal{A}}) \geq \Phi_{i+|\textbf{I}|-|\textbf{J}|}(\mathcal{A}),  ~~~ 1 
\leq i \leq |\textbf{J}|.
\end{equation}
\end{theorem}
\begin{proof}
By the sectorial tensor  decomposition $\mathcal{A}= \mathcal{Q}^{H} *_{N} \mathcal{D} *_{N} \mathcal{Q}$ in Theorem \ref{th2_3}, we have $\tilde{\mathcal{A}}=(\mathcal{Q}*_{N} \mathcal{U})^{H} *_{N} \mathcal{D} *_{N} (\mathcal{Q}*_{N} \mathcal{U})$, which means $W(\tilde{\mathcal{A}})\subset W(\mathcal{A})$. Hence, $0\notin W(\tilde{\mathcal{A}})$, i.e., $\tilde{\mathcal{A}}$ is sectorial. On the other hand, using Lemma \ref{th3_1}, we can obtain
\begin{eqnarray*}
   \Phi_{i}(\mathcal{A})
    &=&\max_{\mathcal{M}:\dim \mathcal{M}=i}\min_{\substack{\mathcal{X} \in \mathcal{M}, \\||\mathcal{X}||=1} } \angle (\mathcal{X}^{H} *_{N} \mathcal{A}  *_{N} \mathcal{X} )\\
&\geq&\max_{\mathcal{K}:\dim\mathcal{K}=i}\min_{\substack{\mathcal{Y} \in \mathcal{K}, \\||\mathcal{Y}||=1}}  \angle (\mathcal{Y}^{H} *_{N} \tilde{\mathcal{A}}  *_{N} \mathcal{Y})\\
&=& \Phi_{i}(\tilde{\mathcal{A}}).
\end{eqnarray*}
The inequality above is due to the fact that for any element $\mathcal{Y}$ in $\mathcal{K}$ there exists  an element $\mathcal{X}=\mathcal{U} *_{L} \mathcal{Y}$ in $\mathcal{M}$ of the same size, such that $\mathcal{X}^{H} *_{N} \mathcal{A}  *_{N} \mathcal{X}=\mathcal{Y}^{H} *_{N} \tilde{\mathcal{A}}  *_{N} \mathcal{Y}$. The other half of the inequality \eqref{eq:nov26_1} can be proven in a similar  way.
\end{proof}


Let  $\mathcal{C} \in \mathbb{C}^{(I_{1} \times \cdots \times I_{N})\times (J_{1} \times \cdots \times J_{L})}$ with $|\textbf{J}|<|\textbf{I}|$ be a nonsingular tensor; cf. Definition \ref{def:rank}. $\C$ has a QR factorization \cite{erfanifar2024polar,r12}, i.e., $\mathcal{C}=\mathcal{Q} *_{N} \mathcal{R}$, where $\mathcal{Q} \in \mathbb{C}^{(I_{1} \times \cdots \times I_{N})\times (J_{1} \times \cdots \times J_{L})}$ is a column orthogonal tensor, and $\mathcal{R} \in \mathbb{C}^{(J_{1} \times \cdots \times J_{L})\times (
J_{1} \times \cdots \times J_{L})}$ is a nonsingular upper triangular tensor. By Lemma \ref{th2_4},  $\Phi_{i}(\mathcal{Q}^{H} *_{N} \mathcal{A} *_{N} \mathcal{Q} )=\Phi_{i}(\mathcal{R}^{H} *_{N} \mathcal{Q}^{H} *_{N} \mathcal{A} *_{N} \mathcal{Q} *_{N} \mathcal{R} )=\Phi_{i}(\mathcal{C}^{H} *_{N} \mathcal{A} *_{N} \mathcal{C} )$, and $\mathcal{Q}^{H} *_{N} \mathcal{A} *_{N} \mathcal{Q}$ is also a compression of $\mathcal{A}$. Therefore, Theorem \ref{th3_2} also holds for arbitrary nonsingular tensors, which are not necessarily  column orthogonal.

\begin{corollary}\label{co3_2_1}
Let $\mathcal{A}\in \mathbb{C}^{(I_{1} \times \cdots \times I_{N})\times (I_{1} \times \cdots \times I_{N})}$ be sectorial and $\mathcal{C}\in \mathbb{C}^{(I_{1} \times \cdots \times I_{N})\times (J_{1} \times \cdots \times J_{L})}$ with $|\textbf{J}|<|\textbf{I}|$  be a nonsingular tensor. Denote $\tilde{\mathcal{A}}=\mathcal{C}^{H} *_{N} \mathcal{A} *_{N}\mathcal{C}$. Then $\tilde{\mathcal{A}}$ is also sectorial and its phases satisfy Eq. \eqref{eq:nov26_1}.
\end{corollary}

Theorem \ref{th3_2} give us the restricted intervals of the phases of compressions. But when we choose $J_{1}=I_{1},\cdots,J_{L}=I_{L}$, that is $\mathcal{X}\in \mathbb{C}^{(I_{1} \times \cdots \times I_{N})\times (I_{1} \times \cdots \times I_{L})}$ ($L<N$), the following theorem further explains that the ``$\leq$'' sign can be taken as equal in some cases.

\begin{theorem}\label{thm:aug1_temp}
Let $\mathcal{A}\in \mathbb{C}^{(I_{1} \times \cdots \times I_{N})\times (I_{1} \times \cdots \times I_{N})}$ be sectorial, and $\mathcal{F}_{N,L}$ be the space of all nonsingular tensors in $\mathbb{C}^{(I_{1} \times \cdots \times I_{N})\times (I_{1} \times \cdots \times I_{L})}$. Then
\begin{equation}\label{eq3_2}
    \max_{\mathcal{X}\in \mathcal{F}_{N,L}} \sum_{i=1}^{I_{1}I_{2}\cdots I_{L}} \Phi_{i}(\mathcal{X}^{H} *_{N} \mathcal{A}  *_{N} \mathcal{X})=\sum_{i=1}^{I_{1}I_{2}\cdots I_{L}} \Phi_{i}(\mathcal{A} ),
\end{equation}
\begin{equation}\label{eq3_3}
    \min_{\mathcal{X}\in \mathcal{F}_{N,L}} \sum_{i=1}^{I_{1}I_{2}\cdots I_{L}} \Phi_{i}(\mathcal{X}^{H} *_{N} \mathcal{A}  *_{N} \mathcal{X})=\sum_{i=l}^{I_{1}I_{2}\cdots I_{N}} \Phi_{i}(\mathcal{A} ),
\end{equation}
where $l=1+\prod_{n=1}^{N}I_{n}-\prod_{n=1}^{L}I_{n}$. 
\end{theorem}

\begin{proof}
    For all nonsingular $\mathcal{X}\in \mathcal{F}_{N,L}$, applying Corollary \ref{co3_2_1} yields
\begin{multline*}
 \sum_{i=1}^{I_{1}I_{2}\cdots I_{L}} \Phi_{i}(\mathcal{X}^{H} *_{N} \mathcal{A} *_{N} \mathcal{X}) \leq 
 \sum_{i=1}^{I_{1}I_{2}\cdots I_{L}} \Phi_{i}(\mathcal{A}).  
\end{multline*}
Assume $\mathcal{A}$ has the sectorial tensor  decomposition $\mathcal{A}= \mathcal{T}^{H} *_{N} \mathcal{D} *_{N} \mathcal{T}$. Define a tensor $\mathcal{I}^{0} \in \mathbb{C}^{(I_{1} \times \cdots \times I_{L})\times (I_{1} \times \cdots \times I_{N})}$ as
\begin{equation*}\label{eq2_5_4}
(\mathcal{I}^{0})_{i_{1}\cdots i_{L}j_{1}\cdots j_{N}} =
\begin{cases}
1, & \text{if } (i_{1},\cdots, i_{L}) = (j_{1},\cdots, j_{L}) \\
  & \text{and } j_{L+1} = \cdots = j_{N} = 1 \\
0, & \text{others.}
\end{cases}
\end{equation*}

Let $\mathcal{X}=\mathcal{T}^{-1} *_{L} (\mathcal{I}^{0})^{H}$. In this case,  $\mathcal{X}$ is a nonsingular tensor. If the eigenvalues of $\mathcal{D}$ are arranged in the decreasing order, then $\sum_{i=1}^{I_{1}I_{2}\cdots I_{L}} \Phi_{i}(\mathcal{X}^{H} *_{N} \mathcal{A}  *_{N} \mathcal{X})= \sum_{i=1}^{I_{1}I_{2}\cdots I_{L}} \Phi_{i}(\mathcal{A})$ by calculation. This proves Eq. (\ref{eq3_2}). The proof of Eq. (\ref{eq3_3}) is similar.
\end{proof}

 
\subsection{Compound spectra and numerical ranges of sectorial tensors}\label{Compound numerical ranges}
In this subsection,  we define compound spectra and compound numerical range of tensors and products of tensors. They will be used  in the study of phases of  product and sum of sectorial tensors in Section \ref{Phases of tensor product and sum}.

\begin{lemma}[\cite{r6}]\label{th4_1}
Given two tensors $\mathcal{A}, \mathcal{B}\in \mathbb{C}^{(I_{1} \times \cdots \times I_{N})\times (I_{1} \times \cdots \times I_{N})}$,  the determinant of the Einstein product $\mathcal{A} *_{N} \mathcal{B}$ satisfies
\begin{equation}\label{eq4_1}
    det(\mathcal{A}*_{N} \mathcal{B})=det(\mathcal{A})det(\mathcal{B}).
\end{equation}
\end{lemma}

In the following, we give the definition of the $k$-th compound spectrum and $k$-th compound numerical range of square tensors.

\begin{definition}\label{def:compound spectrum}
For each $k \in  [|\textbf{I}|]$, the $k$-th compound spectrum of the  tensor $\mathcal{A} \in \mathbb{C}^{(I_{1} \times \cdots \times I_{N})\times (I_{1} \times \cdots \times I_{N})}$ is defined as 
\begin{multline*}
\Lambda_{(k)}(\mathcal{A}) = \left\{ \prod_{m=1}^{k}\lambda_{i_{m}}(\mathcal{A}) : 
1 \leq i_{1} < \cdots < i_{k} \leq |\textbf{I}| \right\}.
\end{multline*}
\end{definition}

\begin{definition}[\cite{r7}]\label{def:compound numerical range}
For each $k \in  [|\textbf{I}|]$, the $k$-th compound numerical range of the tensor $\mathcal{A} \in \mathbb{C}^{(I_{1} \times \cdots \times I_{N})\times (I_{1} \times \cdots \times I_{N})}$ is defined as
\begin{align*}
W_{(k)}(\mathcal{A})=& \left\{ \prod_{m=1}^{k}\lambda_{m}(\tilde{\A}): \tilde{\A} \in \mathbb{C}^{(J_1 \times \cdots \times J_L) \times (J_1 \times \cdots \times J_L)} \right. \\
 & \hspace{0ex} \left.    {\rm is~ an~ arbitrary~ compression~ of~ \A ~{\rm with}~ |\textbf{J}| = k} \right\}.
\end{align*}
\end{definition}

Similarly, one can also define the $k$-th compound angular numerical range. 
\begin{definition}[\cite{r7}]
For each $k \in  [|\textbf{I}|]$, the $k$-th compound angular numerical range of the  tensor $\mathcal{A} \in \mathbb{C}^{(I_{1} \times \cdots \times I_{N})\times (I_{1} \times \cdots \times I_{N})}$ is defined as
\begin{align*}
W_{(k)}'(\mathcal{A})=& \left\{ \prod_{m=1}^{k}\lambda_{m}(\tilde{\A}): \tilde{\A}=\mathcal{X}^{H}*_{N}\A*_{N}\mathcal{X},\right. \\  &\mathcal{X}\in \mathbb{C}^{(I_1 \times \cdots \times I_N) \times (J_1 \times \cdots \times J_L)} \\ & \hspace{0ex} \left.    {\rm is~ an ~arbitrary~ nonsingular~ tensor~ with}~ |\textbf{J}| = k \right\}.
\end{align*}

\end{definition}

We define the product and quotient of two sets as follows.
$$W' _{(k)}(\mathcal{A})W' _{(k)}(\mathcal{B})=\left\{ab,\ a\in W' _{(k)}(\mathcal{A}), b\in  W' _{(k)}(\mathcal{B})\right\},$$
$$W_{(k)}(\mathcal{A})/W_{(k)}(\mathcal{B})=\left\{\frac{a}{b},\ a\in W_{(k)}(\mathcal{A}), b\in  W_{(k)}(\mathcal{B})\right\}.$$

\begin{theorem}\label{th4_2}
Given tensors $\mathcal{A}, \mathcal{B}\in \mathbb{C}^{(I_{1} \times \cdots \times I_{N})\times (I_{1} \times \cdots \times I_{N})}$ with $\mathcal{B}$ being sectorial,  we have
\begin{equation*}
    \Lambda_{(k)}(\mathcal{A} *_{N} \mathcal{B}^{-1}) \subseteq W_{(k)}(\mathcal{A})/W_{(k)}(\mathcal{B}), \quad k \in  [|\textbf{I}|].
\end{equation*}
\end{theorem}
\begin{proof}
    We first assume $\mathcal{A} *_{N} \mathcal{B}^{-1}$ is diagonalizable, that is $\mathcal{A} *_{N} \mathcal{B}^{-1}$ has $|\textbf{I}|$ eigentensors. In this case, we choose $1 \leq i_{1}<i_{2}<\cdots<i_{k} \leq |\textbf{I}|$, and let
    $D=diag\{\lambda_{i_{1}}(\mathcal{A} *_{N} \mathcal{B}^{-1}),\lambda_{i_{2}}(\mathcal{A} *_{N} \mathcal{B}^{-1}),\cdots,\lambda_{i_{k}}(\mathcal{A} *_{N} \mathcal{B}^{-1})\}\in\mathbb{C}^{k\times k}$. Let $\mathcal{X}\in \mathbb{C}^{(I_{1} \times \cdots \times I_{N})\times (k)}$  be the tensor composed of corresponding eigentensors such that
\begin{equation*}
    \mathcal{X}^{H} *_{N} (\mathcal{A} *_{N} \mathcal{B}^{-1}) =D *_{1} \mathcal{X}^{H}.
\end{equation*}
Note that $\mathcal{X}$ has tensor polar decomposition $\mathcal{X}=\mathcal{U} *_{1} P$; cf.\cite{erfanifar2024polar,r12}, where $\mathcal{U}\in \mathbb{C}^{(I_{1} \times \cdots \times I_{N})\times (k)}$ is a column orthogonal tensor and $P \in \mathbb{C}^{k\times k}$ is a positive definite matrix. Consequently,
\begin{equation*}
    \mathcal{U}^{H} *_{N} (\mathcal{A} *_{N} \mathcal{B}^{-1}) =P^{-H} *_{1} D *_{1} P^{H} *_{1} \mathcal{U}^{H}.
\end{equation*}
Post-multiplying both sides by $\mathcal{B} *_{N} \mathcal{U}$, and we get
\begin{equation}\label{eq4}
    \mathcal{U}^{H} *_{N} \mathcal{A} *_{N} \mathcal{U} =P^{-H} *_{1} D *_{1} P^{H} *_{1} (\mathcal{U}^{H} *_{N} \mathcal{B} *_{N} \mathcal{U}).
\end{equation}
Then take the determinants of (\ref{eq4}), by Lemma \ref{th4_1}, we can obtain
\begin{multline*}
    \prod_{i=1}^{k}\lambda_{i}(\mathcal{U}^{H} *_{N} \mathcal{A} *_{N} \mathcal{U}) = \\
    \prod_{m=1}^{k}\lambda_{i_m}(\mathcal{A} *_{N} \mathcal{B}^{-1})\prod_{i=1}^{k}\lambda_{i}(\mathcal{U}^{H} *_{N} \mathcal{B} *_{N} \mathcal{U}).
\end{multline*}
The claim follows by dividing both sides by $\prod_{i=1}^{k}\lambda_{i}(\mathcal{U}^{H} *_{N} \mathcal{B} *_{N} \mathcal{U})$. Then we finish the proof for the diagonalizable case.

When  $\mathcal{A} *_{N} \mathcal{B}^{-1}$ is not diagonalizable, we choose a sequence $\{\mathcal{A}_{i} \}$ with limit $\mathcal{A}$ such that for all $i,\ \mathcal{A}_{i} *_{N} \mathcal{B}^{-1}$ is diagonalizable. It follows that
\begin{equation*}
    \Lambda_{(k)}(\mathcal{A}_{i} *_{N} \mathcal{B}^{-1}) \subseteq W_{(k)}(\mathcal{A}_{i})/W_{(k)}(\mathcal{B}).
\end{equation*}

Give a column orthogonal tensor  $\mathcal{X} \in\mathbb{C}^{(I_{1} \times \cdots \times I_{N})\times (J_{1} \times \cdots \times J_{L})}$ and $ |\textbf{J}|=k$. By the continuity of eigenvalues, sending $i \rightarrow\infty$ yields $\prod_{i=1}^{k}\lambda_{i}(\mathcal{X}^{H} *_{N} \mathcal{A}_{i} *_{N} \mathcal{X}) \rightarrow \prod_{i=1}^{k}\lambda_{i}(\mathcal{X}^{H} *_{N} \mathcal{A} *_{N} \mathcal{X})$, and thus $W_{(k)}(\mathcal{A}_{i}) \rightarrow W_{(k)}(\mathcal{A})$. Similarly, $\Lambda_{(k)}(\mathcal{A}_{i} *_{N} \mathcal{B}^{-1})  \rightarrow \Lambda_{(k)}(\mathcal{A} *_{N} \mathcal{B}^{-1})$. Therefore, in th case that $\mathcal{A} *_{N} \mathcal{B}^{-1}$ is not diagonalizable, the result also holds.
\end{proof}

Note that when $\mathcal{B}=\mathcal{I}$, $W_{(k)}(\mathcal{B})$ has only one element, i.e., $W_{(k)}(\mathcal{B})=\{ 1 \}$. For this case, we have the relationship between the $k$-th compound spectrum and the $k$-th compound numerical range.
\begin{corollary}\label{co4_2}
For each $k \in  [|\textbf{I}|]$, the $k$-th compound numerical range of the tensor  $\mathcal{A}$ satisfies $  \Lambda_{(k)}(\mathcal{A} ) \subseteq W_{(k)}(\mathcal{A})$.
\end{corollary}

Theorem \ref{th4_2} tells us the compound spectra of $\Lambda_{(k)}(\mathcal{A} *_{N} \mathcal{B}^{-1})$.  The next theorem gives those of $\Lambda_{(k)}(\mathcal{A} *_{N} \mathcal{B})$.

\begin{theorem}\label{th4_3}
Let $\mathcal{A}, \mathcal{B}\in \mathbb{C}^{(I_{1} \times \cdots \times I_{N})\times (I_{1} \times \cdots \times I_{N})}$ be two tensors, and $\mathcal{B}$ be sectorial. Then
\begin{equation*}
    \Lambda_{(k)}(\mathcal{A} *_{N} \mathcal{B}) \subseteq W' _{(k)}(\mathcal{A})W' _{(k)}(\mathcal{B}).
\end{equation*}
\end{theorem}
\begin{proof}
    By Theorem \ref{th4_1}, we have $\Lambda_{(k)}(\mathcal{A} *_{N} \mathcal{B}^{-1}) \subseteq W_{(k)}(\mathcal{A})/W_{(k)}(\mathcal{B})$.  Hence we only need to prove that
\begin{equation}\label{eq4_2}
    1/W_{(k)}(\mathcal{B}^{-1}) \subseteq W' _{(k)}(\mathcal{B}).
\end{equation}   
Let $c \in 1/W_{(k)}(\mathcal{B}^{-1})$. By Definition \ref{def:compound numerical range}, there exists a column orthogonal tensor $\mathcal{X} \in\mathbb{C}^{(I_{1} \times \cdots \times I_{N})\times (J_{1} \times \cdots \times J_{L})}$ with $ |\textbf{J}|=k$, such that
\begin{multline*}
c = \prod_{i=1}^{k}\frac{1}{\lambda_{i}(\mathcal{X}^{H} *_{N} \mathcal{B}^{-1} *_{N} \mathcal{X})} = \\
\prod_{i=1}^{k}\frac{1}{\lambda_{i}((\mathcal{B}^{-1} *_{N} \mathcal{X})^{H} *_{N} \mathcal{B}^{H} *_{N} (\mathcal{B}^{-1} *_{N} \mathcal{X}))}.
\end{multline*} 
Let $\mathcal{Y}=\mathcal{B}^{-1} *_{N} \mathcal{X}$. Then $\mathcal{Y}$ is a nonsingular tensor. Noting that $c=\frac{|c|^{2}}{\bar{c}}$, we have,
\begin{multline*}
c = |c|^{2}\prod_{i=1}^{k}\lambda_{i}(\mathcal{Y}^{H} *_{N} \mathcal{B}^{H} *_{N} \mathcal{Y})^{H} \\
= |c|^{2}\prod_{i=1}^{k}\lambda_{i}(\mathcal{Y}^{H} *_{N} \mathcal{B} *_{N} \mathcal{Y}) \in W' _{(k)}(\mathcal{B}).
\end{multline*}
Therefore, we establish Eq. (\ref{eq4_2}).
\end{proof}

\subsection{Phases of  product and sum of sectorial tensors}\label{Phases of tensor product and sum}
 In this subsection, we  derive the relationship between the phases of  $\mathcal{A} *_{N} \mathcal{B}$ and those of sectorial tensors $\mathcal{A} $ and $ \mathcal{B}$. To begin, we bring in the definition of majorized vectors to measure the size of a vector.
\begin{definition}[\cite{r2}]
Let $x,y \in \mathbb{R}^{n}$ be two vectors. The elements in vector $x$ are arranged  from the largest to the smallest as $x_{1},x_{2},\cdots,x_{n}$.  The elements in vector $y$ are arranged  in the same way. Then, $x$ is said to be majorized by $y$, denoted by $x \prec y$, if
\begin{equation*}
\sum_{i=1}^{k}x_{i} \leq \sum_{i=1}^{k}y_{i},\  k=1,2,\cdots,n-1, \ \text{and}\ \sum_{i=1}^{n}x_{i} = \sum_{i=1}^{n}y_{i}.
\end{equation*}
\end{definition}

Define $\gamma(\mathcal{A}):=\frac{\bar{\Phi}(\mathcal{A})+\underline{\Phi}(\mathcal{A})}{2} \in (-\pi,\pi]$, and call it the {\it phase center} of the sectorial tensor $\mathcal{A}$

\begin{theorem}\label{th5_1}
Let $\mathcal{A}, \mathcal{B}\in \mathbb{C}^{(I_{1} \times \cdots \times I_{N})\times (I_{1} \times \cdots \times I_{N})}$ be two sectorial tensors. If $
\angle \lambda(\mathcal{A} *_{N} \mathcal{B})$ takes values in $(\gamma(\mathcal{A})+\gamma(\mathcal{B})-\pi,\  \gamma(\mathcal{A})+\gamma(\mathcal{B})+\pi)$, 
then
\begin{equation}\label{eq5_1}
    \angle \lambda(\mathcal{A} *_{N} \mathcal{B}) \prec \Phi(\mathcal{A})+\Phi(\mathcal{B}).
\end{equation}
\end{theorem}
\begin{proof}
    Let $\hat{\mathcal{A}}=e^{-\imath\gamma(\mathcal{A})}\mathcal{A}$ and $\hat{\mathcal{B}}=e^{-\imath\gamma(\mathcal{B})}\mathcal{B}$. Then both $\hat{\mathcal{A}}$ and $\hat{\mathcal{B}}$ are sectorial with $\gamma(\hat{\mathcal{A}})=\gamma(\hat{\mathcal{B}})=0$. Therefore, $\angle \lambda(\hat{\mathcal{A}} *_{N} \hat{\mathcal{B}})$ takes value in $(-\pi,\pi)$, and $\Phi_{i}(\hat{\mathcal{A}})=\Phi_{i}(\mathcal{A})-\gamma(\mathcal{A})$, $\Phi_{i}(\hat{\mathcal{B}})=\Phi_{i}(\mathcal{B})-\gamma(\mathcal{B})$, $\angle \lambda_{i}(\hat{\mathcal{A}} *_{N} \hat{\mathcal{B}})=\angle \lambda_{i}(\mathcal{A} *_{N} \mathcal{B})-\gamma(\mathcal{A})-\gamma(\mathcal{B})$ hold for all $i\in |\mathbf{I}|$. Therefore, Eq. (\ref{eq5_1}) holds if and only if 
\begin{equation*}
    \angle \lambda(\hat{\mathcal{A}} *_{N} \hat{\mathcal{B}}) \prec \Phi(\hat{\mathcal{A}})+\Phi(\hat{\mathcal{B}}).
\end{equation*}

Without loss of generality, we assume $\gamma(\mathcal{A})=\gamma(\mathcal{B})=0$. According to Definition \ref{def:compound spectrum}, $\prod_{i=1}^{k} \lambda_{i}(\mathcal{A} *_{N} \mathcal{B}) \in \Lambda_{(k)}(\mathcal{A} *_{N} \mathcal{B}) $. Hence, from Theorem \ref{th4_3} it follows that $$\prod_{i=1}^{k} \lambda_{i}(\mathcal{A} *_{N} \mathcal{B}) \in  W' _{(k)}(\mathcal{A})W' _{(k)}(\mathcal{B}).$$ 
Consequently, there exist two nonsingular tensors $\mathcal{X}, \mathcal{Y} \in \mathbb{C}^{(I_{1} \times \cdots \times I_{N})\times (J_{1} \times \cdots \times J_{L})}$ with  $|\textbf{J}|=k $, such that
\begin{equation*}
    \prod_{i=1}^{k} \lambda_{i}(\mathcal{A} *_{N} \mathcal{B})=\prod_{i=1}^{k} \lambda_{i}(\mathcal{X}^{H} *_{N} \mathcal{A} *_{N} \mathcal{X})\prod_{i=1}^{k} \lambda_{i}(\mathcal{Y}^{H} *_{N} \mathcal{B} *_{N} \mathcal{Y}).
\end{equation*}

Since $\gamma(\mathcal{A})=\gamma(\mathcal{B})=0$, all the phases of $\mathcal{A}$ and $\mathcal{B}$ are in $(-\frac{\pi}{2},\frac{\pi}{2})$, and hence $\angle \lambda(\hat{\mathcal{A}} *_{N} \hat{\mathcal{B}})$ takes value in $(-\pi,\pi)$. We have 
\begin{align*}
   &\quad\sum_{i=1}^{k} \angle \lambda_{i}(\mathcal{A} *_{N} \mathcal{B})\\
   &= \sum_{i=1}^{k} \angle \lambda_{i}(\mathcal{X}^{H} *_{N} \mathcal{A} *_{N} \mathcal{X})+ \sum_{i=1}^{k} \angle \lambda_{i}(\mathcal{Y}^{H} *_{N} \mathcal{B} *_{N} \mathcal{Y}) \\
   &= \sum_{i=1}^{k} \Phi_{i}(\mathcal{X}^{H} *_{N} \mathcal{A} *_{N} \mathcal{X})+ \sum_{i=1}^{k} \Phi_{i}(\mathcal{Y}^{H} *_{N} \mathcal{B} *_{N} \mathcal{Y}) \\
   &\leq \sum_{i=1}^{k} \Phi_{i}(\mathcal{A}) + \sum_{i=1}^{k} \Phi_{i}(\mathcal{B}).
\end{align*}
When $k=|\textbf{I}|$, the unequal sign can be taken as the equal sign due to the sectorial tensor  decomposition in Theorem \ref{th2_3}. The proof is completed.
\end{proof}

Given two real numbers $\alpha,\beta$ satisfying $\beta-\alpha<\pi$,  define $\mathcal{C} [\alpha,\beta]$ to be a set of sectorial tensors  such that all their phases are in the open interval $(\alpha,\beta)$, i.e.,
\begin{multline*}
    \mathcal{C} [\alpha,\beta] = \left\{ 
 \mathcal{A} \in \mathbb{C}^{(I_{1} \times \cdots \times I_{N}) \times (I_{1} \times \cdots \times I_{N})} : 
 \mathcal{A} \text{ is sectorial} \right. \\
 \left. \text{and } \bar{\Phi}(\mathcal{A}) \leq \beta,\ \underline{\Phi}(\mathcal{A}) \geq \alpha \right\}.
\end{multline*}

The next theorem extends the matrix case {\rm \cite{r9}}, which gives us the rough evaluation of the phases of $\mathcal{A}+\mathcal{B}$.

\begin{theorem}\label{th6_1} 
Give sectorial tensors   \\ $\mathcal{A}, \mathcal{B} \in \mathbb{C}^{(I_{1} \times \cdots \times I_{N})\times (I_{1} \times \cdots \times I_{N})}$  and real numbers $\alpha,\beta$  such that $\beta-\alpha<\pi$,  if $\mathcal{A}, \mathcal{B} \in \mathcal{C} [\alpha,\beta]$, then $\mathcal{A}+ \mathcal{B}\in \mathcal{C} [\alpha,\beta]$.
\end{theorem}
\begin{proof}
Since $\mathcal{A}, \mathcal{B} \in \mathcal{C} [\alpha,\beta]$ and $\beta-\alpha<\pi$, there exists an open half plane containing both $W(\mathcal{A})$ and $W(\mathcal{B})$. By geometry, $W(\mathcal{A}+\mathcal{B})$ is also contained in this half plane, and thus $\mathcal{A}+\mathcal{B}$ is sectorial. Moreover, note that if $|\angle a-\angle b|<\pi$, then $\min \{\angle a,\angle b\}<\angle (a+b)<\max \{\angle a,\angle b\}$. Therefore, by Lemma \ref{th3_1},
\begin{align*}
    &\bar{\Phi}(\mathcal{A}+\mathcal{B}) \\
    &= \max_{\substack{\mathcal{X} \in \mathbb{C}^{I_{1} \times \cdots \times I_{N}} \\ \|\mathcal{X}\|=1}}  
    \angle (\mathcal{X}^{H} *_{N} \mathcal{A} *_{N} \mathcal{X} + \mathcal{X}^{H} *_{N} \mathcal{B} *_{N} \mathcal{X}) \\
    &\leq \max_{\substack{\mathcal{X} \in \mathbb{C}^{I_{1} \times \cdots \times I_{N}} \\ \|\mathcal{X}\|=1}} 
    \max \{\angle (\mathcal{X}^{H} *_{N} \mathcal{A} *_{N} \mathcal{X}), \\
    &\quad \angle (\mathcal{X}^{H} *_{N} \mathcal{B} *_{N} \mathcal{X})\} \\
    &= \max \{\bar{\Phi}(\mathcal{A}), \bar{\Phi}(\mathcal{B}) \} \\
    &\leq \beta,
\end{align*}
and
\begin{align*}
    & \underline{\Phi}(\mathcal{A}+\mathcal{B}) \\
    &= \min_{\substack{\mathcal{X} \in \mathbb{C}^{I_{1} \times \cdots \times I_{N}} \\ \|\mathcal{X}\|=1}}  
    \angle (\mathcal{X}^{H} *_{N} \mathcal{A} *_{N} \mathcal{X} + \mathcal{X}^{H} *_{N} \mathcal{B} *_{N} \mathcal{X}) \\
    &\geq \min_{\substack{\mathcal{X} \in \mathbb{C}^{I_{1} \times \cdots \times I_{N}} \\ \|\mathcal{X}\|=1}}
    \min \{\angle (\mathcal{X}^{H} *_{N} \mathcal{A} *_{N} \mathcal{X}), \\
    &\quad \angle (\mathcal{X}^{H} *_{N} \mathcal{B} *_{N} \mathcal{X})\} \\
    &= \min \{\underline{\Phi}(\mathcal{A}), \underline{\Phi}(\mathcal{B}) \} \\
    &\geq \alpha.
\end{align*}
Consequently, $\mathcal{A}+ \mathcal{B}\in \mathcal{C} [\alpha,\beta]$. The proof is completed.
\end{proof}

For all $t \in (0,1)$, it is clear that if $\mathcal{A}, \mathcal{B} \in \mathcal{C} [\alpha,\beta]$, then $t\mathcal{A}, (1-t)\mathcal{B} \in \mathcal{C} [\alpha,\beta]$, and thus $t\mathcal{A}+(1-t)\mathcal{B} \in \mathcal{C} [\alpha,\beta]$. So Theorem \ref{th6_1} has the following corollary.
\begin{corollary}\label{co6_1}
$\mathcal{C} [\alpha,\beta]$ is convex.
\end{corollary}

\subsection{Rank robustness against perturbations}\label{Rank robustness against perturbations}
In this subsection, given two sectorial tensors $\mathcal{A}, \mathcal{B}\in \mathbb{C}^{(I_{1} \times \cdots \times I_{N})\times (I_{1} \times \cdots \times I_{N})}$, we study the robustness of the rank of the tensor  $\mathcal{I} + \mathcal{A}  *_{N} \mathcal{B}$.


Given $\alpha\in[0,\pi)$ and $\mathcal{A} \in \mathbb{C}^{(I_{1} \times \cdots \times I_{N})\times (I_{1} \times \cdots \times I_{N})}$, for each $k\in\{1,\cdots,|\textbf{I}|\}$ define
\begin{multline*}
    \mathcal{C}_{k} [\alpha] = \left\{ 
 \mathcal{A} : \mathcal{A} \text{ is sectorial and } \sum_{i=1}^{k}\Phi_{i}(\mathcal{A}) \leq \alpha, \right. \\
 \left. \sum_{i=|\mathbf{I}|-k+1}^{|\mathbf{I}|}\Phi_{i}(\mathcal{A}) \geq -\alpha \right\}.
\end{multline*}
Thus, if $\mathcal{A}\in \mathcal{C}_{k} [\alpha]$, then the sum of the top $k$ phases is no bigger than $\alpha$, and the sum of the last $k$ small phases is no less than $-\alpha$. In particular, when $k=1$ and $\alpha < -\frac{\pi}{2}$,  $\mathcal{C}_{1} [\alpha]=\mathcal{C} [-\alpha,\alpha]$.

\begin{theorem}\label{th7_1}
Let $\mathcal{A} \in \mathbb{C}^{(I_{1} \times \cdots \times I_{N})\times (I_{1} \times \cdots \times I_{N})}$ be sectorial with phases in $(-\pi,\pi]$. For each  fixed $k\in\{1,\cdots,|\textbf{I}|\}$, $rank(\mathcal{I} + \mathcal{A}  *_{N} \mathcal{B})>|\textbf{I}|-k$ holds for all $\mathcal{B} \in \mathcal{C}_{k} [\alpha]$ if and only if 
\begin{equation*}
    \alpha< \min\left\{ k\pi-\sum_{i=1}^{k}\Phi_{i}(\mathcal{A}),\ \ k\pi+ \sum_{i=|\textbf{I}|-k+1}^{|\textbf{I}|}\Phi_{i}(\mathcal{A})\right\}.
\end{equation*}
\end{theorem}
\begin{proof}
    First, we order the eigenvalues of $\mathcal{A} *_{N} \mathcal{B}$ as $\angle \lambda_{1}(\mathcal{A} *_{N} \mathcal{B}) \geq \angle \lambda_{2}(\mathcal{A} *_{N} \mathcal{B})\geq\cdots\geq\angle \lambda_{|\textbf{I}|}(\mathcal{A} *_{N} \mathcal{B})$. Clearly, $rank(\mathcal{I} + \mathcal{A}  *_{N} \mathcal{B})=|\textbf{I}|-k$ only if $\angle \lambda_{1}(\mathcal{A} *_{N} \mathcal{B}) = \angle \lambda_{2}(\mathcal{A} *_{N} \mathcal{B})=\cdots=\angle \lambda_{k}(\mathcal{A} *_{N} \mathcal{B})=\pi$. 
    
 For sufficiency, by Theorem \ref{th5_1} and the definition of $\mathcal{C}_{k} [\alpha]$, for all $\mathcal{B} \in \mathcal{C}_{k} [\alpha]$ we can obtain
\begin{multline*}
    \sum_{i=1}^{k}\angle \lambda_{i}(\mathcal{A} *_{N} \mathcal{B}) \leq \sum_{i=1}^{k}(\Phi_{i}(\mathcal{A}) + \Phi_{i}(\mathcal{B})) \\
    \leq \alpha + \sum_{i=1}^{k}\Phi_{i}(\mathcal{A}) < k\pi.
\end{multline*}
Therefore,
 \begin{equation*}
    \angle \lambda_{k}(\mathcal{A} *_{N} \mathcal{B})< \pi,
\end{equation*}
that is $rank(\mathcal{I} + \mathcal{A}  *_{N} \mathcal{B})>|\textbf{I}|-k$.

For necessity,  by contradiction suppose that  $\alpha\leq k\pi-\sum_{i=1}^{k}\Phi_{i}(\mathcal{A})$. Let $\mathcal{A}=\mathcal{T}^{H} *_{N} \mathcal{D} *_{N} \mathcal{T}$ be a sectorial tensor  decomposition of $\A$. Construct $\mathcal{E}\in \mathbb{C}^{(I_{1} \times \cdots \times I_{N})\times (I_{1} \times \cdots \times I_{N})}$ is a diagonal tensor, where $P(\mathcal{E})=\{e_{1},e_{2},\cdots,e_{|\textbf{I}|}\}$ satisfies
\begin{equation*}
    |e_{i}|=1, \quad for\ i=1,\cdots,k,
\end{equation*}
\begin{equation*}
    \Phi_{i}(\mathcal{A})+ \angle e_{i}=\pi, \quad for\ i=1,\cdots,k,
\end{equation*}
\begin{equation*}
    e_{i}=1, \quad for\ i=k+1,\cdots,|\textbf{I}|.
\end{equation*}
Define $\mathcal{B}=\mathcal{T}^{-1} *_{N} \mathcal{E} *_{N} \mathcal{T}^{-H}$. It is clear that $\mathcal{B}$ is also sectorial and $\sum_{i=1}^{k}\Phi_{i}(\mathcal{B})=\sum_{i=1}^{k}\angle e_{i}\leq\alpha$, $\sum_{i=|\textbf{I}|-k+1}^{|\textbf{I}|}\Phi_{i}(\mathcal{B})\geq 0 \geq-\alpha$. At this time,
\begin{equation*}
    \mathcal{A}  *_{N} \mathcal{B}=\mathcal{T}^{H} *_{N} \mathcal{D} *_{N}\mathcal{E} *_{N}\mathcal{T} ^{-H}
\end{equation*}
has k eigenvalues at -1, that is $rank(\mathcal{I} + \mathcal{A}  *_{N} \mathcal{B})=|\textbf{I}|-k$, which contradicts to the conditions.

Similarly, suppose to the contraposition that $\alpha\leq k\pi+\sum_{i=|\textbf{I}|-k+1}^{|\textbf{I}|}\Phi_{i}(\mathcal{A})$ can lead to the contradiction. Then we finish the proof.
\end{proof}



The following is an immediately consequence of Theorem \ref{th7_1} by setting  $k=1$ there.
\begin{corollary}
    Let $\mathcal{A} \in \mathbb{C}^{(I_{1} \times \cdots \times I_{N})\times (I_{1} \times \cdots \times I_{N})}$ be sectorial with phases in $(-\pi,\pi]$. Then $\mathcal{I} + \mathcal{A}  *_{N} \mathcal{B}$ is invertible for all $\mathcal{B} \in \mathcal{C} [-\alpha,\alpha]$ if and only if 
\begin{equation*}
    \alpha< \min\left\{ \pi-\Phi_{1}(\mathcal{A}),\ \ \pi+ \Phi_{|\textbf{I}|}(\mathcal{A})\right\}.
\end{equation*}
\end{corollary}


\section{Applications in multilinear control}
\subsection{Small phase theorem for sectorial tensors}\label{Small Phase Theorem}

In this subsection, we present a small phase theorem for sectorial tensors.

First, we will review the unfolding process and block tensors under the Einstein product. For a given sequence of tensor dimensions $\textbf{I}=(I_{1},\cdots,I_{N})$ and a corresponding vector of indices $\textbf{i}=(i_{1},\cdots,i_{N})$, define
\begin{equation*}
    ivec(\textbf{i},\textbf{I}):=i_{1}+\sum_{k=2}^{N}(i_{k}-1)\prod_{j=1}^{k-1}I_{j}.
\end{equation*}
The unfolding of a given tensor $\mathcal{A}=(a_{i_{1}\cdots i_{N}j_{1}\cdots j_{M}}) \in \mathbb{C}^{(I_{1} \times \cdots \times I_{N})\times (J_{1} \times \cdots \times J_{M})}$ to a matrix is defined as an isomorphic map\cite{r12}, \cite{wang2025algebraic}
\begin{eqnarray}
    \phi:\mathbb{C}^{(I_{1} \times \cdots \times I_{N})\times (J_{1} \times \cdots \times J_{M})} &\rightarrow& \mathbb{C}^{|\textbf{I}|\times |\textbf{J}|} \label{chen2024phase}\\
     \mathcal{A}=(a_{i_{1}\cdots i_{N}j_{1}\cdots j_{M}}) &\mapsto& A=(A_{ivec(\textbf{i},\textbf{I})ivec(\textbf{j},\textbf{J})}) \nonumber.
\end{eqnarray}
The isomorphic map $\phi$ enjoys the following properties which can be easily verified.

\begin{lemma} \label{lem:isomorphism}
   Given tensors $\mathcal{A} \in \mathbb{C}^{(I_{1} \times \cdots \times I_{M})\times (K_{1} \times \cdots \times K_{N})}$ and $\mathcal{B} \in \mathbb{C}^{(K_{1} \times \cdots \times K_{N})\times (J_{1} \times \cdots \times J_{L})}$,

\bed
\item[(1)] $\phi(\mathcal{A}*_{N} \mathcal{B}) = \phi(\mathcal{A})\phi(\mathcal{B})$;
\item[(2)] $\phi(\mathcal{A}^{H})=\phi(\mathcal{A})^{H}$;
\item[(3)] For all $\lambda\in\mathbb{C}$, $\phi(\lambda\mathcal{A})=\lambda\phi(\mathcal{A})$;
\item[(4)] If $\mathcal{A}$ is a diagonal tensor, then $\phi(\mathcal{A})$ is a diagonal matrix;
\item[(5)] If $\mathcal{A}$ is a sectorial tensor, then $\phi(\mathcal{A})$ is a sectorial matrix;
\item[(6)] Let $\A$ be square. Then $\lambda$ is an eigenvalue of $\A$ if and only if it is an eigenvalue of $\phi(\A)$.
\item[(6)] Assume $\A$ is sectorial. Then $\bar{\Phi}(\A) = \bar{\Phi}(\phi(\A))$ and $\underline{\Phi}(\A) = \underline{\Phi}(\phi(\A))$.
\eed
\end{lemma}

In the following, we construct bigger tensors from smaller ones. Here we adopt a compact concatenation approach \cite{chen2019multilinear, chen2021multilinear}  to construct block tensors.  
\begin{definition}[$n$-mode block tensor \cite{chen2021multilinear}]
\label{Def: n-modeblocktensor}
	Let $\mathcal{A}, \mathcal{B} \in \mathbb{C}^{(I_{1} \times \cdots \times I_{N})\times (J_{1} \times \cdots \times J_{N})}$. For each $n = 1,\dots, N$, the $n$-mode row block tensor  concatenated by $\mathcal{A}$ and $\mathcal{B}$, denoted by $\left[\begin{matrix}
		\mathcal{A} & \mathcal{B}
	\end{matrix}\right]_n \in \mathbb{C}^{(I_{1} \times \cdots I_{n} \times \cdots \times I_{N})\times (J_{1} \times \cdots \times 2J_{n}\cdots \times J_{N})}$, is defined element-wise  as
    \begin{multline*}
\left(\begin{bmatrix}
		\mathcal{A} & \mathcal{B}
	\end{bmatrix}_n\right)_{i_1\cdots i_N j_1\cdots j_N} = \\
\begin{cases}
\mathcal{A}_{i_1\cdots i_{n}\cdots i_N j_1\cdots j_{n}\cdots j_N}, \\
\quad i_k = 1,\dots , I_k,\ j_k = 1,\dots , J_k,\ \forall k, \\[1ex]
\mathcal{B}_{i_1\cdots i_{n}\cdots i_N j_1\cdots (j_{n}-J_{n})\cdots j_N}, \\
\quad i_k = 1,\dots , I_k,\ \forall k,\ j_k = 1,\dots, J_k \\
\quad \text{for } k\ne n \text{ and } j_n = J_n + 1, \dots, 2J_n.
\end{cases}
\end{multline*}
	The $n$-mode column block tensor  is $ \left[\begin{smallmatrix}
		\mathcal{A} \\
		\mathcal{B}
	\end{smallmatrix}\right]_n := \left[\begin{smallmatrix}
		\mathcal{A}^{\top} & \mathcal{B}^{\top}
	\end{smallmatrix}\right]_n^{\top}$.
\end{definition}

Clearly. $\left[\begin{matrix}
		\mathcal{A} & \mathcal{B}
	\end{matrix}\right]_1 $ of two tensors $\mathcal{A}, \mathcal{B} \in \mathbb{C}^{(I_{1} \times \cdots \times I_{N})\times (J_{1} \times \cdots \times J_{N})}$ is a direct generalization of $\left[\begin{matrix}
		A & B
	\end{matrix}\right] $ of matrices $A,B$ of the same row numbers.

We also denote by 
$\left[\begin{smallmatrix}
	\mathcal{A} &  \mathcal{B} \\
	\mathcal{C} &  \mathcal{D}
\end{smallmatrix}\right]_n =\left[\begin{smallmatrix}
	\left[\begin{smallmatrix}
		\mathcal{A} & \mathcal{B}
\end{smallmatrix}\right]_n\\
	\left[\begin{smallmatrix}
		\mathcal{C} & \mathcal{D}
	\end{smallmatrix}\right]_n
\end{smallmatrix}\right]_n$, the $n$-mode block tensor concatenated by the $n$-mode row block tensors $\left[\begin{smallmatrix}
	\mathcal{A} & \mathcal{B}
\end{smallmatrix}\right]_n$ and $\left[\begin{smallmatrix}
	\mathcal{C} & \mathcal{D}
\end{smallmatrix}\right]_n$.

Under the Einstein product, block tensors enjoy  properties similar to their matrix counterparts.
\begin{proposition}[\cite{chen2021multilinear}] \label{Pro: blocktensor}
	Let $\mathcal{A}, \mathcal{B} \in \mathbb{C}^{(I_{1} \times \cdots \times I_{N})\times (J_{1} \times \cdots \times J_{N})}$, $ \mathcal{C}, \mathcal{D} \in \mathbb{C}^{ (J_{1} \times \cdots \times J_{N})\times(I_{1} \times \cdots \times I_{N})}$. The following properties of block tensors hold for all $n=1,\dots,N$.
	\begin{enumerate} \rm
		\item
		$	\begin{bmatrix}
			\mathcal{P}* \mathcal{A} & \mathcal{P}*\mathcal{B}
		\end{bmatrix}_n = \mathcal{P}* \begin{bmatrix} \mathcal{A} & \mathcal{B}	\end{bmatrix}_n
        holds\ for\ all\\  tensors\  \mathcal{P}\ with\ compatible \  dimensions.$ 
        \vspace{2pt}    
		\item 
		$	\begin{bmatrix}
			\mathcal{C}* \mathcal{Q} \\
			\mathcal{D}* \mathcal{Q}
		\end{bmatrix}_n = \begin{bmatrix}
			\mathcal{C} \\
			\mathcal{D}
		\end{bmatrix}_n* \mathcal{Q}\ holds\ for\ all\  tensors\\  \mathcal{Q}\ with\ compatible \  dimensions$.
        \vspace{2pt}
		\item
		$		\begin{bmatrix}
			\mathcal{A} &  \mathcal{B} 
		\end{bmatrix}_n* \begin{bmatrix}
			\mathcal{C} \\
			\mathcal{D}
		\end{bmatrix}_n = 
		\mathcal{A}* \mathcal{C} + \mathcal{B}* \mathcal{D}. $
	\end{enumerate}
\end{proposition}

Ragnarsson and Van Loan \cite{MR2902676} studied the unfolding patterns of block tensors: the subblocks of a tensor can be mapped to contiguous blocks in the unfolding matrix through a series of row and column permutations. Specifically, let $s = qr$, where $q,r$ are positive integers.  A perfect shuffle permutation $\Pi _{q,r} \in \mathbb{R}^{s\times s}$ is defined by
\[ \Pi_{q,r} \mathbf{z} = \begin{bmatrix}
	z_{1:r:s} \\
	z_{2:r:s} \\
	\vdots  \\
	z_{r:r:s}
\end{bmatrix}, \ \ \forall  \ \mathbf{z}\in \mathbb{C}^{s}. \]

 The following lemma is an immediate consequence of \cite[Theorem 3.3]{MR2902676}. 

\begin{lemma} [\cite{wang2025algebraic}] \label{Lemma: unfolding}
	Given even-order square tensors $\mathcal{A}, \mathcal{B}, \mathcal{C}, \mathcal{D} \in \mathbb{C}^{ (I_{1} \times \cdots \times I_{N})\times(I_{1} \times \cdots \times I_{N})}$, there exists a permutation matrix $P = Q_N\cdots Q_2Q_1$ such that 	
    \begin{equation}\label{eq:blcking}
     \phi\left( \begin{bmatrix}
		\mathcal{A} & \mathcal{B} \\
		\mathcal{C} & \mathcal{D}
	\end{bmatrix}_n \right) =  P \begin{bmatrix}
		\phi(\mathcal{A})  & \phi(\mathcal{B}) \\
		\phi(\mathcal{C}) & \phi(\mathcal{D})
	\end{bmatrix} P^{\top}, 
    	\end{equation}
	where $Q_k = I_{2I_1\cdots I_N}$ for $k \le n$, and $Q_k = I_{I_{k+1}\cdots I_N}\otimes \Pi _{I_k,2} \otimes I_{I_1\cdots I_{k-1}}$ for $k \ge n+1$.
\end{lemma}

 Recall that $\mathcal{RH}_\infty$ is the space of all proper and real-rational stable transfer matrices; see for example \cite[pp. 100]{ZDG96}. Let $\mathcal{RH}^{(I_{1} \times \cdots \times I_{N})\times (I_{1} \times \cdots \times I_{N})}_{\infty}$ be the set of all  real rational proper stable transfer tensors, i.e., each entry of $\mathcal{G}(s)\in\mathcal{RH}^{(I_{1} \times \cdots \times I_{N})\times (I_{1} \times \cdots \times I_{N})}_{\infty}$ is a proper real-rational function with all the roots of its denominator lying in the open left-half plane. Similar to the matrix case (\cite[pp. 100]{ZDG96}), for a transfer tensor $\mathcal{G}(s) \in \mathcal{RH}^{(I_{1} \times \cdots \times I_{N})\times (I_{1} \times \cdots \times I_{N})}_{\infty}$, we define its $H_{\infty}$  norm   as
 \begin{equation} \label{eq:H infty norm}
 	\| \mathcal{G} \|_{\infty}:= \sup_{\omega \in \mathbb{R}} \| \mathcal{G}(\iota\omega) \|_2,
 \end{equation}
 where $\|\cdot \|_2$ is the spectral norm, namely the largest singular value  $\sigma_{\rm max}(\cdot)$.

As all the coefficients  of a transfer tensor  $\mathcal{G}(s)\in\mathcal{RH}^{(I_{1} \times \cdots \times I_{N})\times (I_{1} \times \cdots \times I_{N})}_{\infty}$ are real, $\mathcal{G}(\imath \omega)$ is conjugate symmetric, i.e.,  $    \mathcal{G}(-\imath \omega)=\overline{\mathcal{G}(\imath \omega)}$. Thus,  the numerical ranges $W(\mathcal{G}(\imath \omega))$ and $W(\mathcal{G}(-\imath \omega))$ are symmetric about the real axis.

Recently, a class of continuous-time   multi-linear (MLTI) systems of the form   
\begin{subequations} \label{MLTI-continuous}
  \begin{align} 
		\dot{\mathcal{X}}(t) =&\; \mathcal{A}*_N \mathcal{X}(t) + \mathcal{B}*_N\mathcal{U}(t), \\
		\mathcal{Y}(t) =&\;  \mathcal{C}*_N\mathcal{X}(t)+ \mathcal{D}*_N\mathcal{U}(t),
\end{align}  
\end{subequations}
has been studied in \cite{wang2025algebraic}, where the state $\mathcal{X}(t)$, input $\mathcal{U}(t)$, and output $\mathcal{Y}(t) $ are continuous-time  $N$th-order tensor processes,  while $\mathcal{A},\mathcal{B}, \mathcal{C}, \mathcal{D}$ are constant $(2N)$th-order squared tensors.  The transfer  tensor of the continuous-time MLTI system \eqref{MLTI-continuous} is defined as  
\begin{equation} \label{eq:tf}
	\mathcal{G}(s) =    \mathcal{D} +  \mathcal{C}*_N (s \mathcal{I}- \mathcal{A})^{-1}*_N  \mathcal{B}. 
\end{equation}

\begin{figure}[htbp!]
	\centering
	\includegraphics[width=0.45\textwidth]{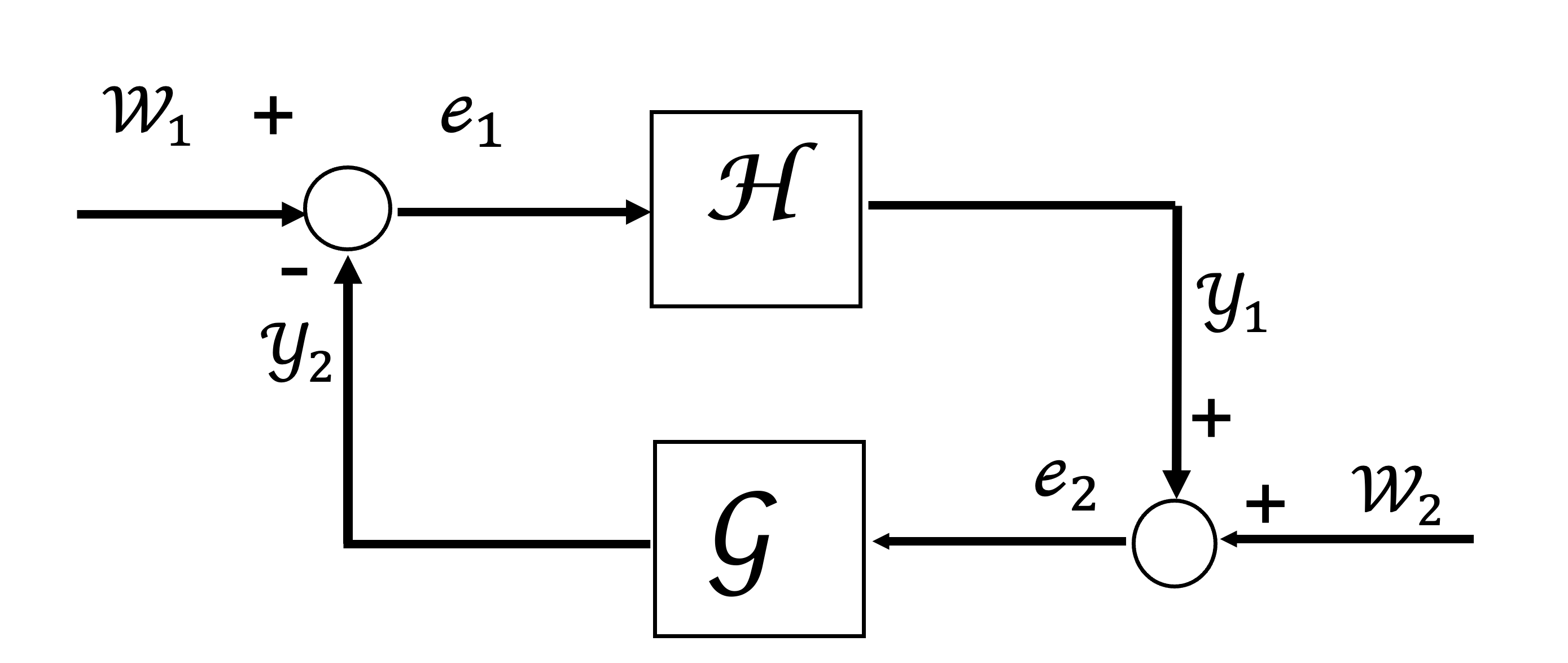}
	\caption{Closed-loop stability of the feedback tensor system $\mathcal{G}\#\mathcal{H}$.}
	\label{fig_small_phase} 
\end{figure}


Given $\mathcal{G},\mathcal{H} \in \mathcal{RH}^{(I_{1} \times \cdots \times I_{N})\times (I_{1} \times \cdots \times I_{N})}_{\infty}$, the feedback interconnection of $\mathcal{G}$ and $\mathcal{H}$ is shown in Fig. \ref{fig_small_phase}. Clearly, the transfer tensor form $(w_1,w_2)$ to $(e_1,e_2)$ is given by the Gang of Four tensor \cite[Chapter 5]{ZDG96}
\begin{align}
   & \mathcal{G}\#\mathcal{H} \nonumber\\
   =&
   \left[
   \begin{smallmatrix}
   \mathcal{I}- \mathcal{G}*_N (\mathcal{I}+\mathcal{H}*_N \mathcal{G})^{-1} *_N \mathcal{H} &   -\mathcal{G}*_N (\mathcal{I}+\mathcal{H}*_N \mathcal{G})^{-1}  \\
   (\mathcal{I}+\mathcal{H}*_N \mathcal{G})^{-1} *_N \mathcal{H}  &  (\mathcal{I}+\mathcal{H}*_N \mathcal{G})^{-1}
  \end{smallmatrix}
  \right]_{1}. \label{eq:feedback}
\end{align}
Here, the subscript ``1'' means $n=1$ in the tensor blocking in Eq. \eqref{eq:blcking}. Therefore, the feedback system is stable if  $\mathcal{G}\#\mathcal{H}\in\mathcal{RH}^{(2I_{1} \times I_{2} \cdots \times I_{N})\times (2I_{1} \times I_{2} \cdots \times I_{N})}_{\infty}$.

\bmrk
In \cite[Figure 1]{wang2025algebraic}, $\mathcal{Y}_2$ is a positive feedback, similar to \cite[Figure 5.2]{ZDG96}, while in  Fig. \ref{fig_small_phase}  negative feedback of $\mathcal{Y}_2$ is adopted. Due to this, we have the term $(\mathcal{I}+\mathcal{H}*_N \mathcal{G})^{-1}$ in Eq. \eqref{eq:feedback}, instead of  $(\mathcal{I}-\mathcal{H}*_N \mathcal{G})^{-1}$ used in \cite{wang2025algebraic}. Such difference makes no difference in the small gain theorem where gains are concerned, but it will make a different in the small phase theorem to be derived, as an additional angle $\pi$ will be introduced.
\emrk 


 By means of the tensor blockings in Definition \ref{Def: n-modeblocktensor}, the closed-loop transfer tensor in Eq. \eqref{eq:feedback} is exact the same form as matrix blocking. In the tensor algebra, there are other types of tensor blockings, such as those in Ref. \cite{MR3479026,ke2016numerical}. Specifically, given two tensors $\mathcal{A}=(a_{i_{1}\cdots i_{N}j_{1}\cdots j_{M}}) \in \mathbb{C}^{(I_{1} \times \cdots \times I_{N})\times (J_{1} \times \cdots \times J_{M})}$ and $\mathcal{B}=(b_{i_{1}\cdots i_{N}k_{1}\cdots k_{M}}) \in \mathbb{C}^{(I_{1} \times \cdots \times I_{N})\times (K_{1} \times \cdots \times K_{M})}$, a row block tensor is denoted by
\begin{equation*}
    \begin{pmatrix}
    \mathcal{A} & \mathcal{B}
    \end{pmatrix}
    \in \mathbb{C}^{(I_{1} \times \cdots \times I_{N})\times (\beta_{1} \times \cdots \times \beta_{M})},
\end{equation*}
    where $\beta_{i}=J_{i}+K_{i}, \forall i \in [M]$. For each $\forall i\in [M]$, denote the set  $\Gamma_{i}=\{J_{i}+1,J_{i}+2,\cdots,J_{i}+K_{i}\}$.  Then the elements in  the block tensor $\begin{pmatrix}
    \mathcal{A} & \mathcal{B}
    \end{pmatrix}$ are
\begin{multline*}
\begin{pmatrix}
    \mathcal{A} & \mathcal{B}
\end{pmatrix}_{i_{1}\cdots i_{N}l_{1}\cdots l_{M}} =\\
\begin{cases}
a_{i_{1}\cdots i_{N}l_{1}\cdots l_{M}}, & 
\begin{aligned}
&\text{if } (i_{1},\cdots, i_{N}) \in [I_{1}] \times \cdots \times [I_{N}], \\
&\quad (l_{1},\cdots, l_{M}) \in [J_{1}] \times \cdots \times [J_{M}],
\end{aligned} \\
b_{i_{1}\cdots i_{N}l_{1}\cdots l_{M}}, & 
\begin{aligned}
&\text{if } (i_{1},\cdots, i_{N}) \in [I_{1}] \times \cdots \times [I_{N}], \\
&\quad (l_{1},\cdots, l_{M}) \in \Gamma_{1} \times \cdots \times \Gamma_{M},
\end{aligned} \\
0, & \text{others.}
\end{cases}
\end{multline*}

Similarly,  given tensors $\mathcal{A} \in \mathbb{C}^{(I_{1} \times \cdots \times I_{N})\times (J_{1} \times \cdots \times J_{M})}$ and $\mathcal{C} \in \mathbb{C}^{(L_{1} \times \cdots \times L_{N})\times (J_{1} \times \cdots \times J_{M})}$, a column block tensor is denoted by
\begin{equation*}
    \begin{pmatrix}
    \mathcal{A} \\
    \mathcal{C}
    \end{pmatrix}
    \in \mathbb{C}^{(\alpha_{1} \times \cdots \times \alpha_{N})\times (J_{1} \times \cdots \times J_{M})},
\end{equation*}
where $\alpha_{i}=I_{i}+L_{i}, \forall i=[N]$. In fact, the column block tensor and the row block tensor have relation
\begin{equation*}
    \begin{pmatrix}
    \mathcal{A} \\
    \mathcal{C}
    \end{pmatrix}=
    \begin{pmatrix}
    \mathcal{A}^{T} & \mathcal{C}^{T}
    \end{pmatrix}^{T}\in \mathbb{C}^{(\alpha_{1} \times \cdots \times \alpha_{N})\times (J_{1} \times \cdots \times J_{M})}
\end{equation*}

Let $\mathcal{D} \in \mathbb{C}^{(L_{1} \times \cdots \times L_{N})\times (K_{1} \times \cdots \times K_{M})}$. By means of the row and column tensor blockings as defined above,  we can form the normal block tensor as  
\begin{equation} \label{eq:blocking_Wei}
    \begin{pmatrix}
    \mathcal{A} & \mathcal{B} \\
    \mathcal{C} & \mathcal{D}
    \end{pmatrix}
    \in \mathbb{C}^{(\alpha_{1} \times \cdots \times \alpha_{N})\times (\beta_{1} \times \cdots \times \beta_{M})}.
\end{equation}

If the tensor blocking \eqref{eq:blocking_Wei} is adopted, the closed-loop tensor transfer will \textbf{not} have the form of Eq. \eqref{eq:feedback} and consequently, many nice results in linear systems theory are not applicable. Thus, it is crucial to choose appropriate tensor blockings.  Indeed,  with the aid of the tensor blocking \eqref{eq:blcking} and the isomorphism \eqref{chen2024phase}, a tensor version of the small gain theorem was recently developed in \cite[Section 4.3]{wang2025algebraic}; specifically, the feedback system $\mathcal{G}\#\mathcal{H}$ is stable if
\begin{equation*}
\|\G\|_\infty \|\mathcal{H}\|_\infty < 1.
\end{equation*}

Using the tensor phase theory  developed above, we could get a tensor version of the small phase theorem,  generalizing the matrix case recently established in \cite{chen2024phase} .

Similar to the matrix case \cite{chen2024phase}, we define frequency-wise sectorial transfer tensors.
\begin{definition}
A transfer tensor \newline $\mathcal{G}\in\mathcal{RH}^{(I_{1} \times \cdots \times I_{N})\times (I_{1} \times \cdots \times I_{N})}_{\infty}$ is said to be frequency-wise sectorial if $\mathcal{G}(\imath \omega)$ is sectorial for all $\omega \in[-\infty,\infty]$.
\end{definition}

\begin{theorem} \label{thm:small_phase_theorem}
    Given frequency-wise sectorial tensors $\mathcal{G},\mathcal{H} \in\mathcal{RH}^{(I_{1} \times \cdots \times I_{N})\times (I_{1} \times \cdots \times I_{N})}_{\infty}$, the feedback system $\mathcal{G}\#\mathcal{H}$ in Fig. \ref{fig_small_phase} is stable if
\begin{equation}\label{spt condition}
    \bar{\Phi}(\mathcal{G}(\imath \omega))+\bar{\Phi}(\mathcal{H}(\imath \omega))<\pi,\quad \underline{\Phi}(\mathcal{G}(\imath \omega))+\underline{\Phi}(\mathcal{H}(\imath \omega))>-\pi
\end{equation}
holds for all $\omega\in [-\infty, \infty]$.
\end{theorem}

\begin{proof}
    As $\mathcal{G},\mathcal{H} \in\mathcal{RH}^{(I_{1} \times \cdots \times I_{N})\times (I_{1} \times \cdots \times I_{N})}_{\infty}$,  by Eq. \eqref{eq:feedback}, the feedback system $\mathcal{G}\#\mathcal{H}$ is stable if and only if  $(\mathcal{I}+\mathcal{H}*_{N}\mathcal{G})^{-1}\in \mathcal{RH}^{(I_{1} \times \cdots \times I_{N})\times (I_{1} \times \cdots \times I_{N})}_{\infty}$. Hence, as given in the proof of \cite[Lemma 4.8]{wang2025algebraic},  it suffices to show that $det[\mathcal{I}+\mathcal{G}(s)*_{N}\mathcal{H}(s)]\neq 0$ for all $s\in\mathbb{C}^{+}\cup \{\infty\}$, where $\mathbb{C}^{+}$ denotes the closed right-half plane.  Let $\phi$ be the isomorphic mapping of the unfolding process defined in Eq. \eqref{chen2024phase}. Then by Lemma \ref{lem:isomorphism} we have  that $$\bar{\Phi}(\mathcal{G}(\imath \omega))=\bar{\Phi}(\phi(\mathcal{G})(\imath \omega)), ~  \underline{\Phi}(\mathcal{G}(\imath \omega))=\underline{\Phi}(\phi(\mathcal{G})(\imath \omega)),$$ 
    $$\bar{\Phi}(\mathcal{H}(\imath \omega))=\bar{\Phi}(\phi(\mathcal{H})(\imath \omega)), ~ \underline{\Phi}(\mathcal{H}(\imath \omega))=\underline{\Phi}(\phi(\mathcal{H})(\imath \omega)).$$ 
    Therefore, by Eq. \eqref{spt condition} we have that
    \begin{equation*}
        \bar{\Phi}((\phi(\mathcal{G})(\imath \omega))+\bar{\Phi}((\phi(\mathcal{H})(\imath \omega))<\pi, 
    \end{equation*}
    \begin{equation*}
        \underline{\Phi}((\phi(\mathcal{G})(\imath \omega))+\underline{\Phi}((\phi(\mathcal{H})(\imath \omega))>-\pi
    \end{equation*}
    holds for all $\omega\in\mathbb{R}$. Then by \cite[Thm. 4.1]{chen2024phase},  
$det[I+\phi(\mathcal{H})(s)\phi(\mathcal{G})(s)]\neq 0$ for all $s\in\mathbb{C}^{+}\cup \{\infty\}$. As $\phi$ is an isomorphism,  it preserves invertibility. Thus $det[\mathcal{I}-\mathcal{H}(s)*_{N}\mathcal{G}(s)]\neq 0$ for all $s\in\mathbb{C}^{+}\cup \{\infty\}$.  
The proof is completed. 
\end{proof}

\subsection{Quasi-sectorial and semi-sectorial tensors} \label{sec:quasi-semi}


In this subsection, we study  quasi-sectorial and semi-sectorial tensors. Due to the multi-dimensional nature of tensors, quasi-sectorial and semi-sectorial tensors exhibit delicate features compared with their matrix counterparts.
    
\begin{definition}
    An even-order square tensor $\mathcal{A}$ is   quasi-sectorial  if its field angle $\delta(\mathcal{A})< \pi$, and is  semi-sectorial if $\delta(\mathcal{A})\leq \pi$.
\end{definition}

It can easily verified  that $\mathcal{A}$ is a quasi-sectorial tensor if and only if $\phi(\mathcal{A})$ is a quasi-sectorial matrix and $\mathcal{A}$ is a semi-sectorial tensor if and only if $\phi(\mathcal{A})$ is a semi-sectorial matrix.   As a result,  similar to  the tensor  decomposition Theorem \ref{th2_3}  for sectorial tensors, quasi-sectorial tensors have a tensor decomposition theorem as given below.

\begin{theorem}\label{quasisectorial}
Suppose $\mathcal{A}\in \mathbb{C}^{(I_{1} \times \cdots \times I_{N})\times (I_{1} \times \cdots \times I_{N})}$ is a quasi-sectorial tensor.  If there exists  an index $n \in \{1, \dots, N\}$ and a positive integer $J_n < I_n$ such that  $rank(\mathcal{A}) \leq \frac{J_n}{I_n} |\textbf{I}|$, then $\mathcal{A}$ has a decomposition
    \begin{equation} \label{eq:quasi_decomp}
        \mathcal{A}=\mathcal{U}*_{N} \begin{bmatrix}
    \mathcal{O} & \mathcal{O} \\
    \mathcal{O} & \mathcal{A}_{s}
    \end{bmatrix}_{n}*_{N}
    \mathcal{U}^{H},
    \end{equation}
    where $\mathcal{U}$ is a unitary tensor and the tensor $\mathcal{A}_{s} \in \mathbb{C}^{(I_{1} \times\cdots\times J_{n}\times \cdots \times I_{N})\times (I_{1} \times\cdots\times J_{n}\times \cdots \times I_{N})}$ is  quasi-sectorial.
\end{theorem}
\begin{proof}
By definition, $\phi(\mathcal{A})$ is a quasi-sectorial matrix and therefore admits the decomposition \cite{chen2024phase}
    \begin{equation*}
        \phi(\mathcal{A})=U\begin{bmatrix}
    O & O \\
    O & A_{t}
    \end{bmatrix} U^{H},
    \end{equation*}
where $U$ is a unitary matrix and $A_t$ is a sectorial matrix. It should be noted that by definition of the isomorphism $\phi$,  we may not be able to apply $\phi^{-1}$ to the matrix $A_t$  due to mismatch of dimensions. Fortunately, by grouping some zero blocks with $A_t$   we can always obtain a bigger matrix $A_{s}=\begin{bmatrix}
    O & O \\
    O & A_{t}
    \end{bmatrix}\in\mathbb{C}^{(\frac{J_n}{I_n} |\textbf{I}|)\times(\frac{J_n}{I_n} |\textbf{I}|)}$ for which $\phi^{-1}(A_s)$ is well defined. By doing this we have, 
       \begin{equation*}
        \begin{bmatrix}
    O & O \\
    O & A_{t}
    \end{bmatrix}=\begin{bmatrix}
    O & O \\
    O & A_{s}
    \end{bmatrix}.
    \end{equation*} 
    (It should be noted that  the zeros in in the LHS matrix may have different blockings with those in the RHS matrix.) 
    Clearly, $A_s$ is a quasi-sectorial matrix. According to Lemma \ref{Lemma: unfolding}, there exists a permutation matrix $P$ such that
\begin{equation*}
    \phi\left(\begin{bmatrix}
    \mathcal{O} & \mathcal{O} \\
    \mathcal{O} & \phi^{-1}(A_{s})
    \end{bmatrix}_{n}\right)=P\begin{bmatrix}
    O & O \\
    O & A_{s}
    \end{bmatrix}P^{\top},
\end{equation*}
That is to say
\begin{eqnarray*}
    \phi(\mathcal{A})&=&U\left(P^{\top}\phi\left(\begin{bmatrix}
    \mathcal{O} & \mathcal{O} \\
    \mathcal{O} & \phi^{-1}(A_{s})
    \end{bmatrix}_{n}\right) P\right)U^{H}\\
    &=&(UP^{\top})\phi\left(\begin{bmatrix}
    \mathcal{O} & \mathcal{O} \\
    \mathcal{O} & \phi^{-1}(A_{s})
    \end{bmatrix}_{n}\right) (UP^{\top})^{H}.
\end{eqnarray*}
Setting $\mathcal{U} = \phi^{-1}(UP^{\top})$ and $\mathcal{A}_s = \phi^{-1}(A_s)$ completes the proof.
\end{proof}
\begin{remark}
If $rank(\mathcal{A}) = \frac{J_n}{I_n} |\textbf{I}|$, then the tensor $\mathcal{A}_{s}$ in Theorem \ref{quasisectorial} is sectorial. But in general, the tensor $\A_s$ is Eq. \eqref{eq:quasi_decomp} is quasi-sectorial, instead of sectorial. This is different from the matrix case in \cite{r5}.
\end{remark}

\begin{example}
In specific situations, the selection of the index  $n$ is important to realize the tensor decomposition in Theorem \ref{quasisectorial} for  quasi-sectorial tensors. 
    Consider a tensor $\mathcal{A}\in\mathbb{C}^{(3\times 2)\times (3\times 2)}$ with
    $$\mathcal{A}(1,1,:,:)=\mathcal{A}(1,2,:,:)=\begin{pmatrix}
        0&0&0\\
        0&0&0
    \end{pmatrix},
    $$
    $$\mathcal{A}(1,3,:,:)=\begin{pmatrix}
        0&0&e^{\imath\theta_{1}}\\
        e^{\imath\theta_{1}}&e^{\imath\theta_{1}}&0
    \end{pmatrix},$$
    $$ 
    \mathcal{A}(2,1,:,:)=\begin{pmatrix}
        0&0&e^{\imath\theta_{1}}\\
        e^{\imath\theta_{1}}+e^{\imath\theta_{2}}&e^{\imath\theta_{1}}&e^{\imath\theta_{2}}
    \end{pmatrix},
    $$
    $$\mathcal{A}(2,2,:,:)=\begin{pmatrix}
        0&0&e^{\imath\theta_{1}}\\
        e^{\imath\theta_{1}}&e^{\imath\theta_{1}}+e^{\imath\theta_{3}}&0
    \end{pmatrix},$$
    $$\mathcal{A}(2,3,:,:)=\begin{pmatrix}
        0&0&0\\
        e^{\imath\theta_{2}}&0&e^{\imath\theta_{4}}+e^{\imath\theta_{2}}
    \end{pmatrix},
    $$
    where $\theta_{1},\theta_{2},\theta_{3},\theta_{4}\in(0,\pi)$. It is easy to see $rank(\mathcal{A})=4$.
    Define $\mathcal{A}_{s}\in\mathbb{C}^{(2\times 2)\times (2\times 2)}$ with
    $$\mathcal{A}_{s}(1,1,:,:)=\begin{pmatrix}
        e^{\imath\theta_{1}}&e^{\imath\theta_{1}}\\
        e^{\imath\theta_{1}}&0
    \end{pmatrix},$$
    $$ 
    \mathcal{A}_{s}(1,2,:,:)=\begin{pmatrix}
        e^{\imath\theta_{1}}&e^{\imath\theta_{1}}+e^{\imath\theta_{2}}\\
        e^{\imath\theta_{1}}&e^{\imath\theta_{2}}
    \end{pmatrix},
    $$
    $$\mathcal{A}_{s}(2,1,:,:)=\begin{pmatrix}
        e^{\imath\theta_{1}}&e^{\imath\theta_{1}}\\
        e^{\imath\theta_{1}}+e^{\imath\theta_{3}}&0
    \end{pmatrix},$$
    $$\mathcal{A}_{s}(2,2,:,:)=\begin{pmatrix}
        0&e^{\imath\theta_{2}}\\
        0&e^{\imath\theta_{4}}+e^{\imath\theta_{2}}
    \end{pmatrix}.
    $$
From Example \ref{calphase}, we know that $\mathcal{A}_{s}$ is a sectorial tensor. Choose a unitary tensor $\mathcal{U}\in\mathbb{C}^{(3\times 2)\times (3\times 2)}$ with
    $$\mathcal{U}(1,1,:,:)=\begin{pmatrix}
        1&0&0\\
        0&0&0
    \end{pmatrix}, ~~  \mathcal{U}(1,2,:,:)=\begin{pmatrix}
        0&0&1\\
        0&0&0
    \end{pmatrix},
    $$
    $$\mathcal{U}(1,3,:,:)=\begin{pmatrix}
        0&0&0\\
        1&0&0
    \end{pmatrix}, ~~
     \mathcal{U}(2,1,:,:)=\begin{pmatrix}
        0&1&0\\
        0&0&0
    \end{pmatrix},
    $$
    $$\mathcal{U}(2,2,:,:)=\begin{pmatrix}
        0&0&0\\
        0&1&0
    \end{pmatrix}, ~~
    \mathcal{U}(2,3,:,:)=\begin{pmatrix}
        0&0&0\\
        0&0&1
    \end{pmatrix}.
    $$
    We find that
    \begin{equation*}
        \mathcal{A}=\mathcal{U}*_{N} \begin{bmatrix}
    \mathcal{O} & \mathcal{O} \\
    \mathcal{O} & \mathcal{A}_{s}
    \end{bmatrix}_{1}*_{N}
    \mathcal{U}^{H}.
    \end{equation*}
    This leads to the decomposition in Theorem \ref{quasisectorial}. But for the second index $n=2$, there is no such decomposition.
\end{example}

\begin{theorem}\label{matrix_quasi}
Suppose $\mathcal{A}\in \mathbb{C}^{(I_{1} \times \cdots \times I_{N})\times (I_{1} \times \cdots \times I_{N})}$ is a quasi-sectorial tensor.  Then $\mathcal{A}$ has a decomposition
    \begin{equation}
        \mathcal{A}=\mathcal{U}*_{N}  \mathcal{C}_{s}
    *_{N}
    \mathcal{U}^{H},
    \end{equation}
    where $\mathcal{U}$ is a unitary tensor, $\phi(\mathcal{C}_{s})=\begin{bmatrix}
    O & O \\
    O & C_{s}
    \end{bmatrix}$ and $C_{s}$ is a sectorial matrix.
\end{theorem}
\begin{proof}
    For a quasi-sectorial matrix $\phi(\mathcal{A})$, it admits the decomposition \cite{chen2024phase}
    \begin{equation}\label{eq:nov28}
        \phi(\mathcal{A})=U\begin{bmatrix}
    O & O \\
    O & C_{s}
    \end{bmatrix} U^{H},
    \end{equation}
where $U$ is a unitary matrix and $C_s$ is a sectorial matrix. Define $\mathcal{U}=\phi^{-1}(U)$, and applying the inverse mapping $\phi^{-1}$ to both sides of this equality yields the desired result.
\end{proof}

As shown in Theorem \ref{matrix_quasi}, in general, a quasi-sectorial tensor has many zero eigenvalues. Thus, similar to the matrix case \cite{chen2024phase}, we define the phases of the quasi-sectorial $\A$ as the phases of the sectorial matrix $C_{s}$ in Eq. \eqref{eq:nov28}.


The semi-sectorial tensors also have  tensor decompositions.
\begin{theorem}
Suppose $\mathcal{A}\in \mathbb{C}^{(I_{1} \times \cdots \times I_{N})\times (I_{1} \times \cdots \times I_{N})}$ is a semi-sectorial tensor.  If there exists  an index $n \in \{1, \dots, N\}$ and a positive integer $J_n < I_n$ such that  $rank(\mathcal{A}) \leq \frac{J_n}{I_n} |\textbf{I}|$, then $\mathcal{A}$ has a decomposition
    \begin{equation}
        \mathcal{A}=\mathcal{U}*_{N} \begin{bmatrix}
    \mathcal{O} & \mathcal{O} \\
    \mathcal{O} & \mathcal{A}_{s}
    \end{bmatrix}_{n}*_{N}
    \mathcal{U}^{H},
    \end{equation}
    where $\mathcal{U}$ is a unitary tensor and \newline  $\mathcal{A}_{s} \in \mathbb{C}^{(I_{1} \times\cdots\times J_{n}\times \cdots \times I_{N})\times (I_{1} \times\cdots\times J_{n}\times \cdots \times I_{N})}$ is a tensor with smaller dimensions.
\end{theorem}

The proof is similar to Theorem \ref{quasisectorial}, and thus is omitted.

The next theorem gives us a useful characterization of quasi-sectorial tensors.
\begin{theorem}
    Suppose $\mathcal{A}\in \mathbb{C}^{(I_{1} \times \cdots \times I_{N})\times (I_{1} \times \cdots \times I_{N})}$ is a quasi-sectorial tensor, $\alpha \in (-\pi,\pi]$. Then $\mathcal{A}$ have phases in $(-\frac{\pi}{2}+\alpha, \frac{\pi}{2}+\alpha)$ if and only if there exists $\epsilon > 0$, such that
    \begin{equation}
        e^{-\imath\alpha}\mathcal{A}+e^{\imath\alpha}\mathcal{A}^{H} \geq \epsilon \mathcal{A}^{H}*_{N} \mathcal{A}.
    \end{equation}
\end{theorem}
\begin{proof}
Let $\phi$ denote the isomorphic mapping defined in Eq. \eqref{chen2024phase}. Denote $x = \phi(\mathcal{X})$ and $A = \phi(\mathcal{A})$. By definition, $A$ is a quasi-sectorial matrix. According to \cite[Lemma 2.1]{chen2024phase}, the inequality \begin{equation*}
        x^{H} (e^{-\imath\alpha}A+e^{\imath\alpha}A^{H})x \geq \epsilon x^{H} (A^{H}A)x
    \end{equation*}
holds for all $x \in \mathbb{C}^{|\mathbf{I}|}$. Applying the inverse mapping $\phi^{-1}$ to both sides of this inequality yields the desired result and completes the proof.
\end{proof}

Let $h \in \mathcal{RH}_{\infty}$ be a scalar transfer function with its inverse $h^{-1} \in \mathcal{RH}_{\infty}$. For a real rational proper stable transfer tensor $\mathcal{H}\in\mathcal{RH}^{(I_{1} \times \cdots \times I_{N})\times (I_{1} \times \cdots \times I_{N})}_{\infty}$, define a cone
\begin{multline*}
    \mathcal{C}(h) = \left\{ \mathcal{H} : 
    \bar{\Phi}(\mathcal{H}(\imath \omega)) \leq \frac{\pi}{2} + \angle h(\imath \omega), \right. \\
    \left. \underline{\Phi}(\mathcal{H}(\imath \omega)) \geq -\frac{\pi}{2} + \angle h(\imath \omega), \ 
    \forall \omega \in [0,\infty] \right\}.
\end{multline*}

Based on the properties of  quasi-sectorial tensors discussed above, the following small phase theorem gives a necessary and sufficient condition for the stability of the feedback system $\mathcal{G}\#\mathcal{H}$ in Fig. \ref{fig_small_phase}.

\begin{theorem}
    Let $\mathcal{G}\in\mathcal{RH}^{(I_{1} \times \cdots \times I_{N})\times (I_{1} \times \cdots \times I_{N})}_{\infty}$ and $h \in \mathcal{RH}_{\infty}$ be a scalar transfer function with $h^{-1} \in \mathcal{RH}_{\infty}$. Then the feedback system $\mathcal{G}\#\mathcal{H}$ in Fig. \ref{fig_small_phase}  is stable for all $\mathcal{H}\in \mathcal{C}(h)$ if and only if $\mathcal{G}$ is frequency-wise quasi-sectorial and
    \begin{equation*}
        \bar{\Phi}(\mathcal{G}(\imath \omega))\leq \frac{\pi}{2}-\angle h(\imath \omega),\quad  \forall \omega\in[0,\infty].
        \end{equation*}
         \begin{equation*}
         \underline{\Phi}(\mathcal{G}(\imath \omega))\geq -\frac{\pi}{2}-\angle h(\imath \omega),\quad \forall \omega\in[0,\infty].
    \end{equation*}
\end{theorem}
\begin{proof}
    Let $\phi$ denote the isomorphic mapping defined in Eq. \eqref{chen2024phase}. Then $$\bar{\Phi}(\mathcal{G}(\imath \omega))=\bar{\Phi}(\phi(\mathcal{G}(\imath \omega))),  \ \ \underline{\Phi}(\mathcal{G}(\imath \omega))=\underline{\Phi}(\phi(\mathcal{G}(\imath \omega))).$$ For $\phi(\mathcal{G}(\imath \omega))\in\mathcal{RH}^{(I_{1} \cdots  I_{N})\times (I_{1} \cdots  I_{N})}_{\infty}$, use the small phase theorem with necessity\cite[Theorem 4.2]{chen2024phase}, we finish the proof.
\end{proof}

\section{Conclusions} \label{conclusion}
In this paper, we have studied phase for tensors under the Einstein product. By generalizing the concept of the numerical range from square matrices to even-order square tensors, we defined the phases of a sectorial tensor via a sectorial tensor decomposition. We introduced compression of tensors, and studied the relation between the phases of a sectorial tensor and its compression. We defined compound spectrum and compound numerical ranges for square tensors and studied their properties.  We also investigated phases of product and sum of sectorial tensors and showed that the angles of the eigenvalues of the product of two sectorial tensors are smaller than the sum of their phases. Finally,  we presented small phase theorems for sectorial as well as quasi-sectorial tensors under the Einstein product.  

\begin{ack}                         The authors are grateful to the very helpful discussions with Professor Wei Chen at Peking University.
\end{ack}

\bibliographystyle{abbrv}  

\bibliography{autosam}           

@article{r1,
  title={A note on numerical ranges of tensors},
  author={Chandra Rout, Nirmal and Panigrahy, Krushnachandra and Mishra, Debasisha},
  journal={Linear and Multilinear Algebra},
  volume={71},
  number={16},
  pages={2645--2669},
  year={2023},
  publisher={Taylor \& Francis}
}

@article{r2,
  title={{On the phases of a complex matrix}},
  author={Dan Wang and Wei Chen and Sei Zhen Khong and Li Qiu},
  journal={Linear Algebra and its Applications},
 volume={593},
  pages={155--160},
  year={2020},
}

@article{r3,
  title={{Eigenvalues of the unitary part of a matrix}},
  author={Alfred Horn AND Robert Steinberg},
 volume={9},
  journal={Pacific J. Math},
  pages={541--550},
  year={1959},
}

@article{r4,
  title={{Fischer determinantal inequalities and
  Higham’s Conjecture}},
  author={S.W. Drury},
  journal={Linear Algebra and its Applications},
volume={439},
  pages={3129--3133},
  year={2013},
}

@article{r5,
  title={{Spectral variation under congruence for a nonsingular matrix with 0 on the boundary of its field of values}},
  author={Susana Furtado and Charles R. Johnson},
  journal={Linear Algebra and its Applications},
volume={359},
  pages={67--78},
  year={2003},
}

@article{r6,
  title={{Tensor inversion and its application to the tensor equations with Einstein product}},
  author={M Liang and B Zheng and R Zhao},
  journal={Linear and Multilinear Algebra},
volume={67},
  pages={843--870},
  year={2019},
}

@article{r7,
  title={{Linear operators preserving the decomposable numerical range}},
  author={Marvin Marcus and Ivan Filippenko},
  journal={Linear and Multilinear Algebra},
volume={7},
  pages={27--36},
  year={1979},
}

@article{r9,
  title={A matrix decomposition and its applications},
  author={Fuzhen Zhang},
  journal={Linear and Multilinear Algebra},
  volume={63},
  pages={2033--2042},
  year={2015},
}

@book{ZDG96, 
	author = {Zhou, Kemin and Doyle, John C. and Glover, Keith}, 
	title = {Robust and Optimal Control}, 
	year = {1996}, 
	isbn = {0134565673}, 
	publisher = {Prentice-Hall, Inc.}, 
	address = {USA} 
}

@article{r12,
  title={{Solving Multilinear Systems via Tensor Inversion}},
  author={M. Brazell and N. Li and C. Navasca and C. Tamon},
  journal={SIAM Journal on Matrix Analysis and Applications},
  volume={34},
  number={2},
  pages={542--570},
  year={2013},
}

@article {MR3914335,
    AUTHOR = {Liang, Mao-lin and Zheng, Bing and Zhao, Rui-juan},
     TITLE = {Tensor inversion and its application to the tensor equations
              with {E}instein product},
   JOURNAL = {Linear Multilinear Algebra},
  FJOURNAL = {Linear and Multilinear Algebra},
    VOLUME = {67},
      YEAR = {2019},
    NUMBER = {4},
     PAGES = {843--870}
}

@article{kolda2009tensor,
  title={Tensor decompositions and applications},
  author={Kolda, Tamara G and Bader, Brett W},
  journal={SIAM Review},
  volume={51},
  number={3},
  pages={455--500},
  year={2009},
  publisher={SIAM}
}

@Book{acichocki2009nonnegativematrixandtensor,
  title={Nonnegative Matrix and Tensor Factorizations: Applications to Exploratory Multi-way Data  Analysis and Blind Source Separation},
  author={A. Cichocki and R Zdunek and others},
  year={2009},
  publisher={New York, Wiley}
}

@book{lu2013multilinear,
  title={Multilinear Subspace Learning: Dimensionality Reduction of Multidimensional Data},
  author={Lu, Haiping and Plataniotis, Konstantinos N and Venetsanopoulos, Anastasios},
  year={2013},
  publisher={CRC Press}
}

@book{qi2018tensor,
  title={Tensor Eigenvalues and their Applications},
  author={Qi, Liqun and Chen, Haibin and Chen, Yannan},
  volume={39},
  year={2018},
  publisher={Springer}
}

@article{cui2016eigenvalue,
  title={An eigenvalue problem for even order tensors with its applications},
  author={Cui, Lu-Bin and Chen, Chuan and Li, Wen and Ng, Michael K},
  journal={Linear and Multilinear Algebra},
  volume={64},
  number={4},
  pages={602--621},
  year={2016},
  publisher={Taylor \& Francis}
}

@article{orus2019tensor,
  title={Tensor networks for complex quantum systems},
  author={Or{\'u}s, Rom{\'a}n},
  journal={Nature Reviews Physics},
  volume={1},
  number={9},
  pages={538--550},
  year={2019},
  publisher={Nature Publishing Group UK London}
}

@article{dotson2022deciphering,
  title={Deciphering multi-way interactions in the human genome},
  author={Dotson, Gabrielle A and Chen, Can and Lindsly, Stephen and Cicalo, Anthony and Dilworth, Sam and Ryan, Charles and Jeyarajan, Sivakumar and Meixner, Walter and Stansbury, Cooper and Pickard, Joshua and others},
  journal={Nature Communications},
  volume={13},
  number={1},
  pages={5498},
  year={2022},
  publisher={Nature Publishing Group UK London}
}

@article{zhang2022linear,
  title={Linear quantum systems: a tutorial},
  author={Zhang, Guofeng and Dong, Zhiyuan},
  journal={Annual Reviews in Control},
  volume={54},
  pages={274--294},
  year={2022},
  publisher={Elsevier}
}

@book{nie2023moment,
  title={Moment and Polynomial Optimization},
  author={Nie, Jiawang},
  year={2023},
  publisher={SIAM}
}

@article{letten2019mechanistic,
  title={The mechanistic basis for higher-order interactions and non-additivity in competitive communities},
  author={Letten, Andrew D and Stouffer, Daniel B},
  journal={Ecology Letters},
  volume={22},
  number={3},
  pages={423--436},
  year={2019},
  publisher={Wiley Online Library}
}

@article{cui2023species,
  title={Analysis of higher-order {L}otka-{V}olterra models:
Application of {S}-tensors and the polynomial complementarity problem},
  author={Cui, Shaoxuan and Zhao, Qi and Guofeng Zhang and Jardon-Kojakhmetov, Hildeberto and Cao, Ming},
   pages={},
  year={2025},
  journal={IEEE Transactions on Automatic Control}
}

@book{bonsall1973numerical,
  title={Numerical Ranges II},
  author={Bonsall, Frank F and Duncan, John},
  volume={10},
  year={1973},
  publisher={Cambridge University Press}
}

@book{halmos2012hilbert,
  title={A Hilbert Space Problem Book},
  author={Halmos, Paul Richard},
  volume={19},
  year={2012},
  publisher={Springer Science \& Business Media}
}

@book{roger1994topics,
  title={Topics in Matrix Analysis},
  author={Roger Horn and Charles R Johnson},
  year={1994},
  publisher={Cambridge University Press Cambridge, UK}
}

@article{axelsson1994numerical,
  title={On the numerical radius of matrices and its application to iterative solution methods},
  author={Axelsson, O and Lu, H and Polman, B},
  journal={Linear and Multilinear Algebra},
  volume={37},
  number={1-3},
  pages={225--238},
  year={1994},
  publisher={Taylor \& Francis}
}

@article{eiermann1993fields,
  title={Fields of values and iterative methods},
  author={Eiermann, Michael},
  journal={Linear algebra and its applications},
  volume={180},
  pages={167--197},
  year={1993},
  publisher={Elsevier}
}

@article{ke2016numerical,
  title={Numerical ranges of tensors},
  author={Ke, Rihuan and Li, Wen and Ng, Michael K},
  journal={Linear Algebra and its Applications},
  volume={508},
  pages={100--132},
  year={2016},
  publisher={Elsevier}
}

@article {MR2902676,
    AUTHOR = {Ragnarsson, Stefan and Van Loan, Charles F.},
     TITLE = {Block tensor unfoldings},
   JOURNAL = {SIAM J. Matrix Anal. Appl.},
  FJOURNAL = {SIAM Journal on Matrix Analysis and Applications},
    VOLUME = {33},
      YEAR = {2012},
    NUMBER = {1},
     PAGES = {149--169},
}

@article{erfanifar2024polar,
  title={On Polar Decomposition of Tensors with Einstein Product and a Novel Iterative Parametric Method},
  author={Erfanifar, Raziyeh and Hajarian, Masoud and Sayevand, Khosro},
  journal={Numerical Mamthematics-Theory Methods and Applications},
  volume={17},
  number={1},
  pages={69--92},
  year={2024},
  publisher={GLOBAL SCIENCE PRESS Office B, 9/F, Kings Wing Plaza2, No. 1 On Kwan St~…}
}

@inproceedings{chen2019phase,
  title={Phase analysis of {MIMO LTI} systems},
  author={Chen, Wei and Wang, Dan and Khong, Sei Zhen and Qiu, Li},
  booktitle={2019 IEEE 58th Conference on Decision and Control (CDC)},
  pages={6062--6067},
  year={2019},
  organization={IEEE}
}

@inproceedings{chen2019multilinear,
  title={Multilinear time invariant system theory},
  author={Chen, C. and Surana, A. and Bloch, A. and Rajapakse, I.},
  booktitle={2019 Proceedings of the Conference on Control and its Applications},
  pages={118--125},
  year={2019},
  organization={SIAM}
}

@article{wang2025algebraic,
	title={Algebraic Riccati tensor equations with applications in multilinear control systems},
	author={Wang, Yuchao and Wei, Yimin and Zhang, Guofeng and Chang, Shih Yu},
	journal={SIAM Journal on Control and Optimization},
	volume={63},
	number={5},
	pages={3378--3406},
	year={2025},
	publisher={SIAM}
}

@article{chen2024phase,
  title={A phase theory of multi-input multi-output linear time-invariant systems},
  author={Chen, Wei and Wang, Dan and Khong, Sei Zhen and Qiu, Li},
  journal={SIAM Journal on Control and Optimization},
  volume={62},
  number={2},
  pages={1235--1260},
  year={2024},
  publisher={SIAM}
}

@article{SZQ23,
	title={Computation of the phase and gain margins of {MIMO} control systems},
	author={Srazhidinov, Radik and Zhang, Ding and Qiu, Li},
	journal={Automatica},
	volume={149},
	pages={110846},
	year={2023},
	publisher={Elsevier}
}

@article{MCQ22,
	title={Phases of discrete-time {LTI} multivariable systems},
	author={Mao, Xin and Chen, Wei and Qiu, Li},
	journal={Automatica},
	volume={142},
	pages={110311},
	year={2022},
	publisher={Elsevier}
}

@article{chen2021multilinear,
	title={Multilinear control systems theory},
	author={Chen, C. and Surana, A. and Bloch, A. M. and Rajapakse, I.},
	journal={SIAM Journal on Control and Optimization},
	volume={59},
	number={1},
	pages={749--776},
	year={2021},
	publisher={SIAM}
}

@article{Z14,
	title={Analysis of quantum linear systems' response to multi-photon states},
	author={Zhang, Guofeng},
	journal={Automatica},
	volume={50},
	number={2},
	pages={442--451},
	year={2014},
	publisher={Elsevier}
}

@article{Z17,
	title={Dynamical analysis of quantum linear systems driven by multi-channel multi-photon states},
	author={Zhang, Guofeng},
	journal={Automatica},
	volume={83},
	pages={186--198},
	year={2017},
	publisher={Elsevier}
}

@article {MR3479026,
	AUTHOR = {Sun, Lizhu and Zheng, Baodong and Bu, Changjiang and Wei,
	Yimin},
	TITLE = {Moore-{P}enrose inverse of tensors via {E}instein product},
  journal={Linear and Multilinear Algebra},
	FJOURNAL = {Linear and Multilinear Algebra},
	VOLUME = {64},
	YEAR = {2016},
	NUMBER = {4},
	PAGES = {686--698},
}

@article{cui2025metzler,
	title={On metzler positive systems on hypergraphs},
	author={Cui, Shaoxuan and Zhang, Guofeng and Jard{\'o}n-Kojakhmetov, Hildeberto and Cao, Ming},
	journal={IEEE Transactions on Control of Network Systems},
	year={2025},
	publisher={IEEE}
}

@book{hackbusch2012tensor,
  title={Tensor Spaces and Numerical Tensor Calculus},
  author={Hackbusch, W.},
  series={Springer Series in Computational Mathematics},
  year={2012},
  publisher={Springer}
}

@article{comon2014tensor,
  title={Tensor Decompositions: State of the Art and Applications},
  author={Comon, P.},
  journal={Mathematics in Signal Processing V},
  pages={1--24},
  year={2014},
  publisher={Oxford University Press}
}

@article{braman2010third,
  title={Third-Order Tensors as Linear Operators on a Space of Matrices},
  author={Braman, B.},
  journal={Linear Algebra and Its Applications},
  volume={433},
  number={7},
  pages={1241--1253},
  year={2010}
}

@book{chen2024tensor,
  title={Tensor-Based Dynamical Systems},
  author={Chen, Can},
  year={2024},
  publisher={Springer}
}

@book{qi2017tensor,
  title={Tensor analysis: spectral theory and special tensors},
  author={Qi, Liqun and Luo, Ziyan},
  year={2017},
  publisher={SIAM}
}

@article{cui2024discrete,
	title={On discrete-time polynomial dynamical systems on hypergraphs},
	author={Cui, Shaoxuan and Zhang, Guofeng and Jard{\'o}n-Kojakhmetov, Hildeberto and Cao, Ming},
	journal={IEEE Control Systems Letters},
	volume={8},
	pages={1078--1083},
	year={2024},
	publisher={IEEE}
}

@article{WGZL21,
	title={Quantum context-aware recommendation systems based on tensor singular value decomposition},
	author={Wang, Xiaoqiang and Gu, Lejia and  Lee, Heung-wing and Zhang, Guofeng},
	journal={ Quantum Information Processing},
	volume={20},
	pages={190},
	year={2021},
}

@article{HQZ16,
  title={Computing the geometric measure of entanglement of multipartite pure states by means of non-negative tensors},
  author={Hu, Shenglong and Qi, Liqun and Zhang, Guofeng},
  journal={Physical Review A},
  volume={93},
  number={1},
  pages={012304},
  year={2016},
  publisher={APS}
}

@article{QZN18,
  title={How entangled can a multi-party system possibly be?},
  author={Qi, Liqun and Zhang, Guofeng and Ni, Guyan},
  journal={Physics Letters A},
  volume={382},
  number={22},
  pages={1465--1471},
  year={2018},
  publisher={Elsevier}
}

@article{DPD+25,
  title={Hypergraph reconstruction from dynamics},
  author={Delabays, Robin and De Pasquale, Giulia and Dörfler, Florian and  Zhang, Yuanzhao},
  journal={Nature Communications},
  volume={16},
  number={1},
  pages={2691},
  year={2025},
}

@article{MR3172255,
    AUTHOR = {Gau, Hwa-Long and Wang, Kuo-Zhong and Wu, Pei Yuan},
     TITLE = {Numerical {R}adii for tensor products of operators},
   JOURNAL = {Integral Equations Operator Theory},
  FJOURNAL = {Integral Equations and Operator Theory},
    VOLUME = {78},
      YEAR = {2014},
    NUMBER = {3},
     PAGES = {375--382},
}

@article{HQ18,
  title={Positive definiteness of paired symmetric tensors and elasticity tensors},
  author={Huang, Z.-H. and Qi, L.},
  journal={Journal of Computational and Applied Mathematics},
  volume={338},
  pages={22--43},
  year={2018},
  publisher={Elsevier}
}

@article{einstein1916foundation,
  title={The foundation of the general theory of relativity},
  author={Einstein, A.},
  journal={Annalen Phys},
  volume={49},
  number={7},
  pages={769--822},
  year={1916}
}



\end{document}